\documentclass[11pt]{amsart}

\usepackage{amssymb,amsmath,latexsym}
\usepackage{amsthm}
\usepackage{amsfonts}
\usepackage{enumerate}
\usepackage{booktabs}
\usepackage{mathrsfs}

\usepackage{cite}

\usepackage{hyperref}

\usepackage{graphicx,pgf}
\usepackage{float}

\usepackage{caption}
\usepackage{subcaption}

\usepackage{fancyhdr}

\theoremstyle{plain} 
\newtheorem{thm}{Theorem}[section]
\newtheorem{prop}[thm]{Proposition}
\newtheorem{lemma}[thm]{Lemma}
\newtheorem{cor}[thm]{Corollary}

\theoremstyle{definition}
\newtheorem{defn}[thm]{Definition}

\theoremstyle{remark}

\newtheorem*{remark}{Remark}

\numberwithin{equation}{section}

\newcommand{\mysectionname}{}
\newcommand{\newsection}[1]{\section{#1}\renewcommand{\mysectionname}{\uppercase{#1}}}

\newcommand{\MT}{\mathcal{M}_{\mathbb{T}}}  

\newcommand{\et}[1]{\eta_{#1} }

\newcommand{\bmultg}{\!\begin{array}{c} {\scriptstyle\times} \\[-12pt]\cup\end{array}\!}

\date{April 29, 2014}

\begin{document}
\title{Local limit theorems for multiplicative free convolutions}
\author{Michael Anshelevich}
\address{Michael Anshelevich: Department of Mathematics, Texas A\&M University, College Station, TX 77843-3368, USA}
\email{manshel@math.tamu.edu}
\author{Jiun-Chau Wang}
\address{Jiun-Chau Wang: Department of Mathematics and Statistics, University of Saskatchewan, Saskatoon,
Saskatchewan S7N 5E6, Canada}
\email{jcwang@math.usask.ca}
\author{Ping Zhong}
\address{Ping Zhong: Department of Mathematics, Rawles Hall, 831 East Third Street, Indiana University, Bloomington, Indiana 47405, USA}
\email{pzhong@indiana.edu}



\begin{abstract}

This paper describes the quality of convergence to an infinitely
divisible law relative to free multiplicative convolution. We show that convergence in distribution for products
of identically distributed and infinitesimal free random variables implies superconvergence of their probability
densities to the density of the limit law. Superconvergence to the marginal law of free multiplicative Brownian motion at a specified time is also studied. In the unitary case, the superconvergence to free Brownian motion and that to the Haar measure are shown to be uniform over the entire unit circle, implying further a free entropic limit theorem and a universality result for unitary free L\'{e}vy processes. Finally, the method of proofs on the positive half-line gives rise to a new multiplicative Boolean to free Bercovici-Pata bijection.
\end{abstract}

\maketitle

\newsection{introduction}
Given two probability laws $\mu$ and $\nu$ on the unit circle $\mathbb{T}=\{ e^{i\theta}:-\pi< \theta \leq \pi\}$, recall that their multiplicative free convolution $\mu\boxtimes\nu$
is defined as the distribution of $UV$, where $U$ and $V$ are two
free unitary random variables having distributions $\mu$ and $\nu$,
respectively. If $\mu,\nu$ are supported on the positive half-line
$\mathbb{R}_{+}=\{x:0\leq x < \infty\}$, then the law $\mu\boxtimes\nu$ is the
distribution of $\sqrt{X}Y\sqrt{X}$, where $X$ and $Y$ are now
free positive random variables distributed according to $\mu$ and
$\nu$ (see \cite{BV1993}). The research of limit theorems for $\boxtimes$ has been active in
recent years; for examples, see \cite{BP2000,BB2008a,BWang2008,CG2008}. In
particular, these works address the problem of weak convergence to a given $\boxtimes$-infinitely
divisible law for measures in an infinitesimal triangular array, and the necessary and sufficient conditions for such
a convergence to take place have been found in \cite{BWang2008, CG2008}. The current paper aims to initiate
a new direction in this research by showing that convergence to a
freely infinitely divisible law also happens at the level of probability
density functions under the same conditions of weak convergence.

To be precise, let $\mu_{n}$ be a sequence of probability laws on $\mathbb{R}_{+}$ and $k_{n}\in\mathbb{N}$ such that $\lim_{n\rightarrow\infty}k_{n}=\infty$
and the free convolution
\begin{equation}\nonumber
  \nu_{n}=(\mu_{n})^{\boxtimes k_{n}}=\mu_{n}\boxtimes\mu_{n}\boxtimes\cdots\boxtimes\mu_{n} \quad (k_n \, \, \text{times})
\end{equation}
tends weakly to a non-degenerate law $\nu$ on $\mathbb{R}_{+}$. It is known that such a limit law $\nu$ must be $\boxtimes$-infinitely divisible and that any freely infinitely divisible
law has at most one atom and its non-atomic part is absolutely continuous relative to Lebesgue measure $dx$ (cf. \cite{BB2008a,BB2005}). Furthermore, the density function $d\nu/dx$ is continuous except at finitely many points on $\mathbb{R}_{+}$ and has an analytic continuation at points where it is different from zero. Our first result (Theorem \ref{thm:4.1}) shows that if $I$ is a compact interval on which the density $d\nu/dx$ is continuous and bounded away from zero, then the measures $\nu_{n}$ restricted to $I$ become absolutely continuous in finite
time and their densities extend analytically to a neighborhood of the interval $I$. In addition, these extensions will converge locally uniformly in that neighborhood to the analytic continuation of $d\nu/dx$. An analogous result for measures supported on $\mathbb{T}$ is obtained in Theorem 4.3.

Next, we study convergence to free multiplicative Brownian motions, with an emphasis on the unitary case. Introduced in \cite{BianeBM}, the free unitary Brownian motion is a free multiplicative L\'{e}vy process $\left(U_t\right)_{t \geq 0}$ of unitary random variables that solves the free stochastic differential equation $dU_t=idX_t\,U_t-(1/2)U_t\,dt$, $U_0=I$, where the process $\left(X_t\right)_{t \geq 0}$ is a free semicircular Brownian motion. As shown by Biane \cite{BianeBM}, the distribution of the process $\left(U_t\right)_{t \geq 0}$ can be regarded as the limit of that of Brownian motion with values in the group of $N\times N$ unitary matrices, as the dimension $N\rightarrow \infty$. (Note that the case of approximating free multiplicative Brownian motion on $\mathbb{R}_{+} $ by large Brownian motion on the general linear group is proved recently in \cite{CG2013,Kemp2013}.) Free unitary Brownian motion is related to Voiculescu's liberation process in his theory of free Fisher information and free entropy \cite{DVV1999}, and many of its properties have been intensively studied in the literature \cite{BianeJFA,Zhong1,Zhong2}.

Here we shall be interested in the approximation of the law of $U_t$ by the free convolution powers $\nu_{n}$ on the circle $\mathbb{T}$. A motivation for this is a central limit problem solved in \cite{Wang}, where it was shown that the $k_n$-fold classical multiplicative convolution $\mu_{n}\circledast\mu_{n}\circledast\cdots\circledast\mu_{n}$ converges weakly on $\mathbb{T}$ to a wrapped normal of variance $t$ if and only if the corresponding free convolution $\nu_n$ converges weakly to the law of $U_t$. In Theorem 4.7 we prove, in addition to the analyticity of $d\nu_{n}/d\theta$, that the convergence to the density of $U_t$ in the preceding free multiplicative central limit theorem (MCLT) is actually uniform throughout the entire space $\mathbb{T}$. Such a global convergence result implies further the convergence of free entropy in this free MCLT (see Proposition 5.1), resembling Barron's classical result on the convergence to the normal entropy in the additive CLT on $\mathbb{R}$ \cite{Barron}. Thus, although free and classical MCLTs are equivalent as weak limit theorems, our results show that the convergence in the free limit theorem is in fact much stronger than the one in the commutative case.

We also pursue these ideas in convergence to the Haar measure of the circle group (i.e., the uniform distribution $d\theta/2\pi$ on $\mathbb{T}$). This is based on the perspective that the usual CLT process leads to a universal law of maximum entropy among all laws of finite second moment on the real line, while at the same time the Haar measure has maximum free entropy zero among all laws on the circle. In search of universal laws on the circle, we prove in Theorem 4.4 that the convergence of $d\nu_n/d\theta$ to $1/2\pi$ is also uniform on $\mathbb{T}$ and use this to establish the convergence of free entropy. Moreover, since Theorem 4.4 is proved under a mild condition on the first moment of $\nu_n$, we are able to use it to show further a uniform universality result which states that every non-degenerate free unitary L\'{e}vy process flows to the Haar measure in law, in density, and in free entropy in the long run (Proposition 5.2).

The additive version of our free local limit theorems originates
from the paper \cite{BV1995sup} of Bercovici and Voiculescu, in which this type of convergence was referred as the "superconvergence." The results in \cite{BV1995sup} were extended to measures with unbounded support in \cite{Wang2010} and to the case of freely stable laws in \cite{Kargin2011}. Our proofs in this paper rely on analytic subordination of free convolution and a technique of recasting weak convergence of measures into local uniform convergence of these subordination functions, where the latter has to do with the asymptotics of certain integral transforms (see Section 3). This method works for additive free convolution $\boxplus$ as well, showing that the superconvergence phenomenon can also be found in weak convergence to a general $\boxplus$-infinitely divisible law. However, the subordination techniques presented in this paper are not enough to prove the global uniform convergence result in the additive case. To obtain a strong result of this sort, genuinely new ideas involving a detailed analysis on the boundary behavior of Cauchy transforms are needed. The additive result is currently under preparation and will be published independently.

Finally, note that the aforementioned free and classical MCLTs establish a correspondence between the free Brownian motion and the wrapped normal on $\mathbb{T}$. The additive version of such a correspondence between free and classical infinitely divisible limit laws on $\mathbb{R}$ is known as the "Bercovici-Pata Bijection" \cite{BPata1999}. Also, a multiplicative analogue of the Bercovici-Pata Bijection between free and Boolean limit laws on $\mathbb{T}$ has been studied thoroughly in \cite{Wang}. Having said that, the similar result for the positive half-line has been missing from the literature. As we shall see in Corollary 3.3, our approach to the free local limit theorems through Boolean convolution leads to this missing Boolean to Free Bercovici-Pata Bijection on $\mathbb{R}_{+}$.

The remainder of this paper is organized into four sections. After collecting some preliminary materials in Section 2, we show the recasting of the weak convergence conditions in the next section. Local limit theorems and the convergence to free multiplicative Brownian motions are treated in Section 4. The applications in free entropy and in free unitary L\'{e}vy processes are presented in the last section.

\section{preliminaries}

\subsection{Free and Boolean convolutions on $\mathbb{R}_{+}$}
Let
$\Omega=\mathbb{C}\setminus\mathbb{R}_{+}$. For any probability law
$\mu$ on $\mathbb{R}_{+}$, we define analytic functions $\psi_{\mu},\,\eta_{\mu}:\,\Omega\rightarrow\Omega$
by \[
\psi_{\mu}(z)=\int_{\mathbb{R}_{+}}\frac{xz}{1-xz}\, d\mu(x)\quad\text{and}\quad\eta_{\mu}(z)=\frac{\psi_{\mu}(z)}{1+\psi_{\mu}(z)}.\] Observe that the imaginary part of $-\pi[1+\psi_{\mu}(1/z)]/z$ is the Poisson integral of $\mu$ for any $z$ in the complex upper half-plane $\mathbb{C}^{+}$.
It follows that the measure $\mu$ is completely determined
by the function $\psi_{\mu}$ (and hence by $\eta_{\mu}$).

In this paper the notation $\mathcal{M}_{+}$ means the set of all
Borel probability measures $\mu$ on $\mathbb{R}_{+}$ such that $\mu$ is not the point mass at the origin. For each $\mu\in\mathcal{M}_{+}$, it is known \cite{BB2005} that the $\eta$-transform
of $\mu$ has the following mapping properties: $\eta_{\mu}\left((-\infty,0)\right) \subset (-\infty,0)$, $0=\eta_{\mu}(0^{-})=\lim_{x\rightarrow0,\, x<0}\eta_{\mu}(x)$,
$\eta_{\mu}(\overline{z})=\overline{\eta_{\mu}(z)}$ for $z\in\Omega$,
and
\begin{equation}
\pi>\arg\eta_{\mu}(z)\geq\arg z,\qquad z\in\mathbb{C}^{+}.\label{eq:2.1}
\end{equation}
Here, and throughout this paper, the notation $\arg$ denotes the principal value of the argument
function with the branch cut $(-\infty,0]$ and the range $(-\pi,\pi]$. In fact, these conditions characterize the functions $\eta_{\mu}$ among all analytic maps $\eta:\Omega\rightarrow \mathbb{C}\setminus\{0\}$ with the property $\eta(\overline{z})=\overline{\eta(z)}$ for $z\in\Omega$. The above properties imply that the analytic function \[B_{\mu}(z)= \frac{z}{\eta_{\mu}(z)}\]
is well-defined in $\Omega$. As we shall see, this $B$-transform will play a role in our investigation of limit theorems.

Given $\mu,\nu\in\mathcal{M}_{+}$, the calculation of their \emph{multiplicative
free convolution} $\mu\boxtimes\nu$ involves the compositional inverse of their
$\eta$-transforms. More precisely, the
function $\eta_{\mu}$ is univalent in the left half-plane $i\mathbb{C}^{+}$, whose range $\eta_{\mu}(i\mathbb{C}^{+})$ is a subset of $i\mathbb{C}_{+}$ and contains an interval of the form $(\alpha,0)$ for some $\alpha<0$. Denoting by $\eta_{\mu}^{-1}$
the analytic inverse of $\eta_{\mu}$ and setting $\Sigma_{\mu}(z)=\eta_{\mu}^{-1}(z)/z$, then it is shown in\cite{BV1993} that \[
\Sigma_{\mu\boxtimes\nu}(z)=\Sigma_{\mu}(z)\Sigma_{\nu}(z)\]
for $z$ in some interval $(\beta,0)$, $\beta<0$.

On the other hand, the product of $B$-transforms does not always give rise to a probability measure in $\mathcal{M}_{+}$. This phenomenon has been studied in \cite{B2006} under the framework of Boolean probability theory. Roughly, the main issue here is that the mapping condition \eqref{eq:2.1} is not always preserved under the process of multiplying $B$-transforms. But when it is, it makes sense to give the following
\begin{defn}[\cite{B2006}]\label{def:Boolean}
 Given $\mu,\nu$ in $\mathcal{M}_{+}$, their \emph{multiplicative Boolean convolution} $\mu\bmultg\nu$ is defined as the unique probability law in $\mathcal{M}_{+}$ that satisfies the following identity:\[B_{\mu\bmultg\nu}(z)=B_{\mu}(z)B_{\nu}(z),\quad z\in\Omega.\]\end{defn}
\begin{remark}
If $\mu,\nu$ are probability laws in $\mathcal{M}_{+}$ with
  \begin{enumerate}
    \item $\arg \et{\mu}(z)+\arg \et{\nu}(z)-\arg z <\pi$ for $z\in\Omega\cap\mathbb{C}^+$, and
    \item at least one of the first moments of $\mu$ and $\nu$ is finite,
  \end{enumerate}
then $\mu\bmultg\nu$ is well-defined as a member in $\mathcal{M}_{+}$.
\end{remark}

The set of finite Borel measures supported on a fixed subset of $\mathbb{C}$ is equipped with the topology of weak convergence. Throughout this paper, the symbol $\Rightarrow$ will be used to indicate this weak convergence of measures. The following result characterizes the weak convergence of probability measures on $\mathbb{R}_{+}$ in terms of their $B$- and $\Sigma$-transforms.
\begin{prop}
Let $\nu,\mu_{1},\mu_{2},\cdots$ be measures in $\mathcal{M}_{+}$. Then the weak convergence $\mu_{n} \Rightarrow \nu$ holds if and
only if the sequence $\{B_{\mu_{n}}\}_{n=1}^{\infty}$ converges to $B_{\nu}$ uniformly on compact subsets of the domain $\Omega$ if and only if there exists a closed disc $D\subset i\mathbb{C}^{+}$ with real center such that the functions $\Sigma_{\mu_{n}}$ are defined in $D$ for all $n$ and the sequence $\{\Sigma_{\mu_{n}}\}_{n=1}^{\infty}$ converges to $\Sigma_{\nu}$ uniformly in the disc $D$.
\end{prop}
\begin{proof}
The characterization involving the $\Sigma$-transform can be found in \cite{BV1993}. We shall prove the one with the $B$-transform.
Assume first that $\mu_{n}$ converges weakly to $\nu$. Since $\left|xz\right|<\left|1-xz\right|$
for any $x\in\mathbb{R}_{+}$ and for any $z\in i\mathbb{C}^{+}$, we have $\psi_{\mu_{n}}(z)\rightarrow\psi_{\nu}(z)$
for all $z$ in $i\mathbb{C}^{+}$. Note that this convergence is in fact
uniform over any compact subset of $\Omega$, because $\{\psi_{\mu_{n}}\}_{n=1}^{\infty}$
is a normal family of analytic self-maps of $\Omega$. This
implies the convergence result for $\{ B_{\mu_{n}}\}_{n=1}^{\infty}$,
since $B_{\mu_{n}}(z)=[z+z\psi_{\mu_{n}}(z)]/\psi_{\mu_{n}}(z)$.

Conversely, the local uniform convergence $B_{\mu_{n}}\rightarrow B_{\nu}$ in $\Omega$ shows that the sequence $\psi_{\mu_{n}}$ converges, locally
uniformly on $ i\mathbb{C}^{+}$, to the map $\psi_{\nu}$. Now, let $\varepsilon>0$ be arbitrary
but fixed. Since $\psi_{\nu}(0^{-})=0$, there exists a $y=y(\varepsilon)<0$
such that $\left|\psi_{\nu}(y)\right| < \varepsilon$. Moreover,
the limit $\lim_{n\rightarrow\infty}\psi_{\mu_{n}}(y)=\psi_{\nu}(y)$
implies that there exists $N=N(\varepsilon)>0$ such that \[
\varepsilon\geq\left|\psi_{\mu_{n}}(y)\right|=\int_{\mathbb{R}_{+}}\frac{x\left|y\right|}{1+x\left|y\right|}\, d\mu_{n}(x)\geq2^{-1}\mu_{n}((\left|y\right|^{-1},\infty))\]
for any $n\geq N$, whence the family $\{\mu_{n}\}_{n=1}^{\infty}$
is tight. Therefore, the sequence $\mu_{n}$ tends weakly to the measure
$\nu$, because any weak limit of $\{\mu_{n}\}_{n=1}^{\infty}$ is
uniquely determined by the convergence $\psi_{\mu_{n}}\rightarrow\psi_{\nu}$
in $i\mathbb{C}^{+}$.
\end{proof}

The class of infinitely divisible measures plays a fundamental role
in limit theorems for free random variables. Recall that a measure
$\mu\in\mathcal{M}_{+}$ is said to be \emph{$\boxtimes$-infinitely
divisible} if for every positive integer $n$ there exists a measure
$\mu_{n}\in\mathcal{M}_{+}$ such that $(\mu_{n})^{\boxtimes n}=\mu$. Likewise, the infinite divisibility relative to other convolutions appearing in this paper is defined in the same way.

Infinitely divisible laws can be characterized through their $\Sigma$-transforms \cite{BV1993}. For $\mu\in\mathcal{M}_{+}$, the mapping
properties of the function $\Sigma_{\mu}$ show that there is an open set $V \subset i\mathbb{C}^{+}$, intersecting $(-\infty, 0)$, such that the principal logarithm $u=\log\Sigma_{\mu}=\ln\left|\Sigma_{\mu}\right|+i\arg\Sigma_{\mu}$
is defined locally as an analytic function in $V$. The law $\mu$ is $\boxtimes$-infinitely
divisible \emph{}if and only if the logarithm $u$ can be extended
analytically to the entire domain $\Omega$ (in which case we still denote it
by $u$) with the properties: \[\Sigma_{\mu}(z)=\exp(u(z)), \quad z \in \Omega,\] and $u(\overline{z})=\overline{u(z)}$, $\Im u(z) \leq 0$ for all $z \in \mathbb{C}^{+}$. Moreover, the extension $u$ possesses a unique
Nevanlinna representation: \begin{equation}
u(z)=\gamma+\int_{[0,\infty]}\frac{1+xz}{z-x}\, d\sigma(x),\qquad z\in\Omega,\label{eq:2.2}
\end{equation} where $\gamma\in\mathbb{R}$ and $\sigma$ is a finite Borel measure on the compact space $[0,\infty]$. Conversely, any such an integral $u$ determines a unique $\boxtimes$-infinitely divisible law $\mu$ via the functional equation $\Sigma_{\mu}(z)=\exp(u(z))$. In the sequel, we will write $u=u_{\gamma,\sigma}$ and $\mu=\mu_{\boxtimes}^{\gamma,\sigma}$ to indicate this correspondence which is usually referred to as the \emph{free L\'{e}vy-Hin\v{c}in's formula}.

The study of $\bmultg$-infinite divisibility is somehow simpler than that in
the free case. It turns out that every measure $\mu\in\mathcal{M}_{+}$ is $\bmultg$-infinitely divisible and its $B$-transform admits a
L\'{e}vy-Hin\v{c}in type integral formula (see \cite{B2006}). This
is related to the fact that the function $B_{\mu}$ never
vanishes in the domain $\Omega$, and its principal logarithm
$v=\log B_{\mu}$ is always defined and analytic in $\Omega$. Moreover,
it follows from \eqref{eq:2.1} that $\Im v(z)\leq0$ for $z\in\mathbb{C}^{+}$
and $v((-\infty,0))\subset\mathbb{R}$. However, note that in the Nevanlinna
form
\begin{equation}
v(z)=\gamma+\int_{[0,\infty)}\frac{1+xz}{z-x}\, d\sigma(x),\qquad z\in\Omega,\label{eq:2.3}
\end{equation}
of the function $v$, one has $\sigma(\{\infty\})=0$ (cf. \cite[Lemma 3.3]{HZ}).
As in the free case, we will write $\mu=\mu_{\bmultg}^{\gamma,\sigma}$ to indicate \eqref{eq:2.3}.

\subsection{Convolutions on the unit circle}
Let $\MT$ be the set of all probability measures with non-zero first moment on the unit circle $\mathbb{T}$.
For $\mu\in\MT$, the functions $\psi_{\mu}$ and $\eta_{\mu}$ are defined by the same formulas as in Section 2.1, with the integrals computed over the circle $\mathbb{T}$. The domain of definition for these functions is now the open unit disc $\mathbb{D}$. Note that $|\et{\mu}(z)|\leq |z|$ for all $z\in\mathbb{D}$ and that the derivative $\et{\mu}'(0)$ is equal to $\text{m}\left(\mu\right)$, the first moment (the mean) of $\mu$. Also, if $\eta: \mathbb{D} \rightarrow  \mathbb{C}$ is an analytic map satisfying $\eta(0)=0$, $\eta^{\prime}(0)\neq 0$, and $|\eta(z)| < 1$ for $z\in\mathbb{D}$, then we have $\eta=\eta_{\mu}$ for some measure $\mu \in \MT$. Thus, the function $\et{\mu}$ is always invertible in a neighborhood of the origin, and we further define the $\Sigma$-transform of $\mu$ by \[\Sigma_{\mu}(z)=\frac{\et{\mu}^{-1}(z)}{z}\] for $z$ near the origin, where the value of $\Sigma_{\mu}(0)$ is given by $1/\text{m}\left(\mu\right)$. Given measures $\mu,\nu\in\MT$, the mean of their free convolution $\mu \boxtimes \nu$ is given by \[\text{m}\left(\mu \boxtimes \nu\right)=\text{m}\left(\mu\right)\text{m}\left(\nu\right),\] and we have that $\Sigma_{\mu\boxtimes\nu}(z)=\Sigma_{\mu}(z)\Sigma_{\nu}(z)$ for $z$ in a neighborhood of zero on which all three functions involved are defined.

In contrast with the positive line case, the Boolean convolution of any two probability measures on the circle $\mathbb{T}$ is always well-defined. The following characterization for Boolean convolution was proved in \cite{Franz}: Given two probability measures $\mu,\nu$ on $\mathbb{T}$, their multiplicative Boolean convolution $\mu\bmultg\nu$ is the unique probability measure on $\mathbb{T}$ such that \[z\, \eta_{\mu\bmultg\nu}(z)=\eta_{\mu}(z)\, \eta_{\nu}(z),\quad z\in \mathbb{D}.\]

As shown in \cite{BV1992}, weak convergence of probability measures on $\mathbb{T}$ is equivalent to the local uniform convergence of the corresponding $\eta$- and $\Sigma$-transforms on their domains of definition.

We end this section with a brief discussion of infinite divisibility.
We know from \cite{BV1992} that a measure $\mu\in\MT$ is $\boxtimes$-infinitely divisible if and only if there exist $\alpha\in\mathbb{T}$ and a finite positive Borel measure $\sigma$ on $\mathbb{T}$ such that the $\Sigma$-transform has the form \begin{equation}\label{eq:2.4}
  \Sigma_{\mu}(z)=\alpha \exp \left(\int_{\mathbb{T}}\frac{1+\xi z}{1-\xi z}d\,\sigma(\xi)\right), \quad z\in \mathbb{D}.
\end{equation}
Also, the only $\boxtimes$-infinitely divisible law with zero mean is the Haar measure $d\theta/2\pi$ on $\mathbb{T}$. Likewise, a measure $\nu$ on $\mathbb{T}$  is $\bmultg$-infinitely divisible if and only if either $\nu=d\theta/2\pi$ or else the function  $\eta_{\nu}(z)\neq 0$ for all $z\in\mathbb{D}\setminus\{0\}$. In the latter case, the $B$-transform $B_{\nu}(z)=z/\eta_{\nu}(z)$ is well-defined and analytic in $\mathbb{D}$, where \[B_{\nu}(0)=1/\eta_{\nu}^{\prime}(0)=1/\text{m}\left(\nu\right).\] Furthermore, we can write $B_{\nu}(z)=\exp(u(z))$ for some analytic function $u$ where $\Re u(z) \geq 0$ for all $z \in \mathbb{D}$. In particular, the Herglotz representation of $u$ shows that such a function $B_{\nu}$ is the $\Sigma$-transform
of a $\boxtimes$-infinitely divisible law on $\mathbb{T}$. Conversely, the regularity results proved in \cite[Proposition 3.3]{BB2005} imply that the $\eta$-transform of a freely infinitely divisible law has no zeros in $\mathbb{D}\setminus\{0\}$. Thus, each $\boxtimes$-infinitely divisible law is also $\bmultg$-infinitely divisible.

\subsection{Subordination and global inversion}
An important feature of free convolution is analytic subordination. Namely, given measures $\mu,\nu \in \mathcal{M}_{+}$, there exist unique analytic functions $\omega_1, \omega_2 :\Omega \rightarrow \Omega$ such that the formulas \[\eta_{\mu \boxtimes \nu}(z)=\eta_{\mu}(\omega_1(z))=\eta_{\nu}(\omega_2(z))\] hold in the domain $\Omega$. Likewise, if $\mu$ and $\nu$ are supported on the circle $\mathbb{T}$, then the corresponding subordination functions $\omega_1, \omega_2$ will be two analytic self-maps of the unit disc $\mathbb{D}$ with $\omega_1(0)=0=\omega_2(0)$.

An earlier partial result of the subordination phenomenon was proved by Voiculescu in his work of free entropy, and later Biane proved the subordination result in full generality. (We refer the reader to the article \cite{VoiCoalgebra} for a conceptual understanding of analytic subordination in free probability.) To treat free convolution powers on $\mathbb{R}_{+}$, we require the following version of analytic subordination from the paper \cite{BB2005}. Let $\mu$ be a measure in $\mathcal{M}_{+}$ and let $n\geq 2$ be an integer. Then there exists a unique conformal map $\omega_n:\Omega \rightarrow \Omega$ such that the $\eta$-transform $\eta_{\mu^{\boxtimes n}}=\eta_{\mu} \circ \omega_n$ in $\Omega$, and in this case the map $H_n(z)=z[B_{\mu}(z)]^{n-1}$ serves as an analytic left inverse for the function $\omega_n$ in the set $\Omega$.

In Section 4, we will use the following global inversion result to construct appropriate subordination functions for measures supported on $\mathbb{T}$. This result is a summary of  Theorem 4.4 and Proposition 4.5 in \cite{BB2005}.
\begin{prop}
Let $\Phi:\mathbb{D} \rightarrow \mathbb{C}\cup \{\infty\}$ be a meromorphic function such that $\Phi(0)=0$ and $\left|\Phi(z)\right| \geq \left|z\right|$ for every $z\in \mathbb{D}$.
\begin{enumerate}[$(1)$]
\item There exists a unique continuous and one-to-one map $\omega:\overline{\mathbb{D}}\rightarrow \overline{\mathbb{D}}$ such that $\omega\left(\mathbb{D}\right) \subset \mathbb{D}$, the function $\omega$ is analytic in $\mathbb{D}$, and $\Phi\left(\omega(z)\right)=z$ for $z \in \mathbb{D}$.
\item The range $\omega\left(\mathbb{D}\right)$ is a simply connected domain whose boundary is the simple closed curve $\omega\left(\mathbb{T}\right)$, and we have $\omega\left(\mathbb{D}\right)=\{z\in \mathbb{D}:\left|\Phi(z)\right|<1\}$.
\item A point $z\in \mathbb{D}$ lies on the curve $\omega\left(\mathbb{T}\right)$ if and only if $\left|\Phi(z)\right|=1$.
\item If $\xi\in \omega\left(\overline{\mathbb{D}}\right) \cap \mathbb{T}$, then the entire radius $\{r\xi:0 \leq r<1\}$ is contained in $ \omega\left(\mathbb{D}\right)$.
\item If $\xi \in \mathbb{T}$ is such that $\left|\omega(\xi)\right|<1$, then $\omega$ can be extended analytically to a neighborhood of $\xi$.
\end{enumerate}
\end{prop}

The next result will also be used in Section 4. Its proof is a consequence of the preceding proposition. (See also Theorem 3.2 in \cite{Zhong2}.)
\begin{prop}
Let $\nu$ be a $\boxtimes$-infinitely divisible law in $\mathcal{M}_{\mathbb{T}}$ and $\xi \in \mathbb{T}$. Then the radial line segment $L=\{r\xi:0 <r \leq 1\}$ can only intersect the boundary of $\eta_{\nu}\left(\mathbb{D}\right)$ once.
\end{prop}
\begin{proof}
First, the free L\'{e}vy-Hin\v{c}in's formula \eqref{eq:2.4} shows that Proposition 2.3 can be applied to the map \[\Phi(z)=z\Sigma_{\nu}(z)=\alpha z \exp\left(\int_{x \in (-\pi,\pi]}\frac{1+e^{ix} z}{1-e^{ix} z}\,d\sigma(e^{ix})\right), \quad z\in \mathbb{D},\] where in this case the inverse function $\omega$ is precisely the $\eta$-transform $\eta_{\nu}$. Here the measure $\sigma$ will be assumed to be a non-zero measure, since the desired conclusion holds trivially in the case of $\sigma=0$. Also, we write $d\rho(x)=d\sigma(e^{ix})$ for notational convenience.

Now, suppose to the contrary that the given radius $L$ intersects the boundary $\partial \eta_{\nu}\left(\mathbb{D}\right)$ more than once, say, at $z_1$ and $z_2$ in $\overline{\mathbb{D}}$. If one of these two boundary points lies on the circle $\mathbb{T}$, then Proposition 2.3 (4) shows that the other one must be in the interior $\eta_{\nu}\left(\mathbb{D}\right)$, which is not possible. Hence, both $z_1$ and $z_2$ must belong to the unit disc $\mathbb{D}$, and we have $\left|\Phi(z_1)\right|=1=\left|\Phi(z_2)\right|$ by Proposition 2.3 (3). This last formula shows that the function \[h(r)=1-\log \left|\Phi\left(re^{i\theta}\right)\right|/\log r=\int_{-\pi}^{\pi}\frac{r^2-1}{\log r(1-2r\cos\left(\theta+x\right)+r^2)} \, d\rho(x)\] assumes the same value $1$ at two distinct points $r_1=\left|z_1\right|$ and $r_2=\left|z_2\right|$ in the interval $(0,1)$, where the polar angle $\theta=\arg z_1=\arg z_2$ here is treated as a constant determined by the radius $L$. However, such a result contradicts the fact that the function $h$ is strictly increasing in the interval $(0,1)$. Therefore the line $L$ can only intersect $\partial \eta_\nu\left(\mathbb{D}\right)$ once.
\end{proof}

\section{weak limit theorems and $B$-transforms}
Suppose that $\{\mu_{nj}:n\geq1,\,1\leq j\leq k_{n}\}$ is a triangular
array of probability measures on $\mathbb{R}_{+}$, which satisfies
the following \emph{infinitesimality condition}: \[
\lim_{n\rightarrow\infty}\min_{1\leq j\leq k_{n}}\mu_{nj}((1-\varepsilon,1+\varepsilon))=1\]
for every fixed $\varepsilon>0$. By virtue of this condition, the measures $\mu_{nj}$ all belong to the set $\mathcal{M}_{+}$ as long as $n$ is sufficiently
large, and hence their $B$-transforms are well-defined in the domain
$\Omega$.  Without loss of generality, we shall assume that all infinitesimal arrays in this section are in the set $\mathcal{M}_{+}$.
Given a sequence $c_{n}>0$ and denoting by $\delta_{c_n}$ the point mass concentrated at the point $c_n$, our goal here is to characterize
the weak convergence of the free convolutions \[\nu_{n}=\delta_{c_{n}}\boxtimes\mu_{n1}\boxtimes\mu_{n2}\boxtimes\cdots\boxtimes\mu_{nk_{n}}\] in terms of convergence properties of the corresponding $B$-transforms;
for this will become useful later when we investigate free local limit theorems.

To this purpose, we follow \cite{BWang2008} and introduce a scaled array $\{\mu_{nj}^{\circ}\}_{n,j}$
as follows. Define positive numbers\[
b_{nj}=\exp\left(\int_{1/e}^{e}\log x\, d\mu_{nj}(x)\right)\]
and measures $\mu_{nj}^{\circ}$ by $\mu_{nj}^{\circ}(A)=\mu_{nj}(b_{nj}A)$
for every Borel set $A$ on $\mathbb{R}_{+}$. The resulting array
$\{\mu_{nj}^{\circ}\}_{n,j}$ remains infinitesimal, and we have
\[
\lim_{n\rightarrow\infty}\max_{1\leq j\leq k_{n}}\left|b_{nj}-1\right|=0.\]
Thus, the $B$-transform of $\mu_{nj}^{\circ}$ is defined
in $\Omega$ and \[
B_{\mu_{nj}^{\circ}}(b_{nj}z)=b_{nj}B_{\mu_{nj}}(z),\qquad z\in\Omega.\]

Next, we introduce the auxiliary functions \begin{equation}
g_{nj}(z)=\int_{\mathbb{R}_{+}}\frac{(z-1)(1-x)}{xz-1}\, d\mu_{nj}^{\circ}(x).\label{eq:3.1}\end{equation}
These are analytic functions defined in $\Omega$ with the symmetry
$g_{nj}(\overline{z})=\overline{g_{nj}(z)}$ for $z \in \Omega$, and they map $\mathbb{C}^{+}$ into the lower
half-plane $\mathbb{C}^{-}$. An important feature of the function
$g_{nj}$ is that its real and imaginary parts are asymptotically
comparable in the sense that for any closed disk $D$ in the left half-plane $i\mathbb{C}^{+}$,
there exists a constant $M=M(D)>0$ such that if $n$ is sufficiently
large, we have \begin{equation}
\left|\Re g_{nj}(z)\right|\leq M\left|\Im g_{nj}(z)\right|, \quad z\in D, \quad 1\leq j\leq k_{n}.\label{eq:3.2}\end{equation}
In case $\mu_{nj}(\{0\})=0$,
this estimate was established in Lemma 3.1 of \cite{BWang2008} and it played the key role in proving the free limit theorems therein. The inequality
\eqref{eq:3.2} for general measures follows from that special case
and the decomposition $g_{nj}(z)=\mu_{nj}^{\circ}(\{0\})(1-z)+\widetilde{g_{nj}}(z)$,
where the function $\widetilde{g_{nj}}(z)$ is defined as in the formula
\eqref{eq:3.1} with the integral calculated over the restriction
of the measure $\mu_{nj}^{\circ}$ to the set $(0,\infty)$.

The proof of our main result relies on the following observation.

\begin{lemma}
Let $D\subset i\mathbb{C}^{+}$ be any closed disk with real center
such that all functions $\Sigma_{\mu_{nj}^{\circ}}$ are defined in
$D$. Then there exists a smaller closed disk $D_{1}\subset D$, also
with real center, such that both estimates \textup{(i)} $\log B_{\mu_{nj}^{\circ}}(b_{nj}z)=g_{nj}(z)(1+o(1))$
and \textup{(ii)} $\log\Sigma_{\mu_{nj}^{\circ}}(z)=g_{nj}(z)(1+o(1))$
hold uniformly for $1\leq j\leq k_{n}$ and for $z\in D_{1}$ as $n\rightarrow\infty$.
\end{lemma}
\begin{proof}
The infinitesimality of $\{\mu_{nj}^{\circ}\}_{n,j}$ shows that
the limit $\lim_{n\rightarrow\infty}\psi_{\mu_{nj}^{\circ}}(z)=z/(1-z)$
holds uniformly in $j$ as well as in $z \in D$. Hence, both $B_{\mu_{nj}^{\circ}}(b_{nj}z)$
and $\Sigma_{\mu_{nj}^{\circ}}(z)$ converge to $1$ as $n\rightarrow\infty$ so that their principal logarithms $\log B_{\mu_{nj}^{\circ}}(b_{nj}z)$
and $\log\Sigma_{\mu_{nj}^{\circ}}(z)$ are defined in $D$ when $n$
is large. Since $\log w=(w-1)(1+o(1))$ as $w\rightarrow1$, the current
lemma will be proved if the functions $B_{\mu_{nj}^{\circ}}(b_{nj}z)-1$
and $\Sigma_{\mu_{nj}^{\circ}}(z)-1$ can be approximated uniformly
by $g_{nj}(z)$ up to a factor of order $1+o(1)$.

We first examine the case of $B_{\mu_{nj}^{\circ}}(b_{nj}z)-1$. We
know from linear Taylor approximation that \begin{eqnarray*}
\frac{1}{z}-\frac{1}{\eta_{\mu_{nj}^{\circ}}(z)} & = & \frac{1+z/(1-z)}{z/(1-z)}-\frac{1+\psi_{\mu_{nj}^{\circ}}(z)}{\psi_{\mu_{nj}^{\circ}}(z)}\\
 & = & -\frac{(1-z)^{2}}{z^{2}}\left[\frac{z}{1-z}-\psi_{\mu_{nj}^{\circ}}(z)\right](1+o(1))\end{eqnarray*}
uniformly in $j$ and $z\in D$ as $n\rightarrow\infty$. Since \[
\frac{(1-z)^{2}}{z}\left[\frac{z}{1-z}-\psi_{\mu_{nj}^{\circ}}(z)\right]=\int_{\mathbb{R}_{+}}\frac{(z-1)(1-x)}{xz-1}\, d\mu_{nj}^{\circ}(x)=g_{nj}(z),\]
we conclude that $B_{\mu_{nj}^{\circ}}(z)-1=g_{nk}(z)(1+o(1))$ as
$n\rightarrow\infty$. The estimate (i) then follows from the fact
that $\lim_{n\rightarrow\infty}B_{\mu_{nj}^{\circ}}(b_{nj}z)/B_{\mu_{nj}^{\circ}}(z)=1$
uniformly in $j$ and $z$.

Let us now turn to the second estimate regarding $\Sigma_{\mu_{nj}^{\circ}}(z)-1$.
To prove (ii), observe first that the functions $\eta_{\mu_{nj}^{\circ}}(z)\rightarrow z$ uniformly in $j$ and in $z\in D$ as $n \rightarrow \infty$. Hence,
denoting $D=\{ z:\left|z-t\right|\leq r\}$ and $D_{1}=\{ z:2\left|z-t\right|\leq r\}$,
we deduce that $\eta_{\mu_{nj}^{\circ}}\left(D_{1}\right)\subset D$ and
\[z\Sigma_{\mu_{nj}^{\circ}}(z)-z=\eta_{\mu_{nj}^{\circ}}^{-1}(z)-\eta_{\mu_{nj}^{\circ}}^{-1}(\eta_{\mu_{nj}^{\circ}}(z)),\qquad z\in D_{1},\] whenever $n$ is sufficiently large.
Finally, the linear approximation of $\eta_{\mu_{nj}^{\circ}}^{-1}(z)$
yields the estimate (ii).
\end{proof}
Fix $\gamma\in\mathbb{R}$ and a finite Borel measure $\sigma$ on
$[0,\infty]$, and let $\mu_{\boxtimes}^{\gamma,\sigma}$ be the $\boxtimes$-infinitely
divisible measure associated with these parameters. It is known that
every weak limit of the free convolutions $\{\nu_{n}\}_{n=1}^{\infty}$
must be $\boxtimes$-infinitely divisible (see {[}3]). Our next result
gives the conditions for this convergence.

\begin{thm}\label{thm:3.2}
The following statements are equivalent:
\begin{enumerate}[$(1)$]
\item The sequence $\nu_{n}$ converges weakly to the law $\mu_{\boxtimes}^{\gamma,\sigma}$.
\item The limit \[
\lim_{n\rightarrow\infty}\frac{1}{c_{n}}\prod_{j=1}^{k_{n}}B_{\mu_{nj}}(z)=\Sigma_{\mu_{\boxtimes}^{\gamma,\sigma}}(z)\]
holds locally uniformly in the domain $\Omega$.
\end{enumerate}
\end{thm}
\begin{proof}
By Proposition 2.2, the weak convergence of $\nu_{n}$ to
the law $\mu_{\boxtimes}^{\gamma,\sigma}$ is equivalent to the existence
of a closed disk $D\subset i\mathbb{C}^{+}$ with real center such
that all $\Sigma_{\mu_{nj}}$ are defined in $D$ and the limit \begin{equation}
\lim_{n\rightarrow\infty}\frac{1}{c_{n}}\prod_{j=1}^{k_{n}}\Sigma_{\mu_{nj}}(z)=\Sigma_{\mu_{\boxtimes}^{\gamma,\sigma}}(z)\label{eq:3.3}\end{equation}
holds uniformly in $D$. We shall seek a way of re-casting \eqref{eq:3.3}
in terms of the functions $g_{nj}$ and $B_{\mu_{nj}}$. In order
to do so we recall Lemma 2.3 from \cite{BWang2008}: Let
$\{ z_{nj}\}_{n,j}$ and $\{ w_{nj}\}_{n,j}$ be two triangular arrays
of complex numbers with $\Im w_{nj}\leq0$, and let $\{ r_{n}\}_{n=1}^{\infty}\subset\mathbb{R}$.
Suppose that $z_{nj}=w_{nj}(1+o(1))$ uniformly in $j$ as $n\rightarrow\infty$
and that there exists a constant $M>0$ such that $\left|\Re w_{nj}\right|\leq M\left|\Im w_{nj}\right|$
for large $n$. Then the sequence $r_{n}+\sum_{j=1}^{k_{n}}z_{nj}$
converges if and only if the sequence $r_{n}+\sum_{j=1}^{k_{n}}w_{nj}$
does. Moreover, the two sequences must have the same limit.

Now, take the smaller disk $D_{1}\subset D$ as in Lemma 3.1 to accommodate the uniform approximation of the logarithms $\log\Sigma_{\mu_{nj}^{\circ}}(z)$ and $\log B_{\mu_{nj}^{\circ}}(b_{nj}z)$ by the functions $g_{nj}(z)$. Then the preceding result on numerical arrays, the estimate \eqref{eq:3.2}, and the equation $\Sigma_{\mu_{nj}^{\circ}}(z)=b_{nj}\Sigma_{\mu_{nj}}(z)$ imply that \eqref{eq:3.3} is equivalent to the uniform convergence\begin{equation}
-\log c_{n}+\sum_{j=1}^{k_{n}}\left[g_{nj}(z)-\log b_{nj}\right]\rightarrow \log\Sigma_{\mu_{\boxtimes}^{\gamma,\sigma}}(z)=u_{\gamma,\sigma}(z)\,\ \ \ \ (n\rightarrow \infty)\label{eq:3.4}\end{equation}
in $D_{1}$, where the function
$u_{\gamma,\sigma}$ extends analytically to $\Omega$ as in the free
L\'{e}vy-Hin\v{c}in's formula \eqref{eq:2.2}.

On the other hand, the formula $B_{\mu_{nj}^{\circ}}(b_{nj}z)=b_{nj}B_{\mu_{nj}}(z)$ implies that \eqref{eq:3.4} is further
equivalent to the uniform convergence
\begin{equation}
v_{n}(z)=-\log c_{n}+\sum_{j=1}^{k_{n}}\log B_{\mu_{nj}}(z)\rightarrow u_{\gamma,\sigma}(z)\qquad(n\rightarrow\infty)\label{eq:3.5}
\end{equation} in $D_{1}$, where $\log B_{\mu_{nj}}$ refers to the Boolean L\'{e}vy-Hin\v{c}in's integral formula \eqref{eq:2.3}. Indeed, by re-writing the functions $v_n$ and $u_{\gamma,\sigma}$ in their integral forms, the convergence \eqref{eq:3.5} becomes
\[-\log c_{n}+\gamma_n+\sum_{j=1}^{k_{n}}\int_{\mathbb{R}_{+}}\frac{1+xz}{z-x}\,d\sigma_{nj}(x)\rightarrow \gamma+\sigma(\{\infty\})+\int_{\mathbb{R}_{+}}\frac{1+xz}{z-x}\,d\sigma(x)\qquad(n\rightarrow\infty).\] Since the integrand $(1+xz)/(z-x)$ remains bounded away from zero and infinity for any $z \in D_1$ and $x \in \mathbb{R}_{+}$, we deduce that \[\sum_{j=1}^{k_{n}}d\sigma_{nj}(x) \Rightarrow d\sigma(x) \quad \text{and} \quad \gamma_n - \log c_n \rightarrow \gamma+\sigma(\{\infty\}) \quad (n \rightarrow \infty).\] This implies that \eqref{eq:3.5} actually holds locally uniformly in the entire domain $\Omega$. After applying the exponential to \eqref{eq:3.5}, we get
the equivalence of (1) and (2).
\end{proof}
When the product of the $B$-transforms $B_{\mu_{nj}}$ ($1\leq j\leq k_{n}$) and $1/c_{n}$ is another $B$-transform of a probability
law $\mu_{n}\in\mathcal{M}_{+}$, we may write
\[\mu_{n}=\delta_{c_{n}}\bmultg\mu_{n1}\bmultg\mu_{n2}\bmultg\cdots\bmultg\mu_{nk_{n}}.\]
Then Theorem 3.2 implies a multiplicative analogue of the Boolean to Free Bercovici-Pata
Bijection. Recall that the L\'{e}vy-Hin\v{c}in measure $\sigma$ appearing in
the integral representation \eqref{eq:2.3} does not assign any mass to $\infty$.
\begin{cor}
Fix a real $\gamma$ and a finite Borel measure $\sigma$ on $[0,\infty)$.
The Boolean convolution  $\mu_{n}$ converges weakly to $\mu_{\bmultg}^{\gamma,\sigma}$ if and only if the free convolution $\nu_{n}$ converges weakly to $\mu_{\boxtimes}^{\gamma,\sigma}$.
\end{cor}

We end this section with a remark on the relationship between free and classical limit theorems. By re-writing the function $g_{nj}$ as \[
g_{nj}(z)=\int_{(0,\infty]}\frac{x^2-1}{x^2+1} \,d\mu_{nj}^{\circ}(1/x)+\int_{(0,\infty]}\left[\frac{1+xz}{z-x}\right]\frac{(x-1)^{2}}{x^{2}+1}\, d\mu_{nj}^{\circ}(1/x),\] one can proceed as in \cite{BP2000} to show that the limit \eqref{eq:3.4} yields a different set of necessary and sufficient conditions for the weak convergence $\nu_{n} \Rightarrow \mu_{\boxtimes}^{\gamma,\sigma}$; namely, the convergence $\nu_{n} \Rightarrow \mu_{\boxtimes}^{\gamma,\sigma}$ holds if and only if the sequence of measures \[
d\sigma_{n}(t)=\sum_{j=1}^{k_{n}}\frac{(x-1)^{2}}{x^{2}+1}\, d\mu_{nj}^{\circ}(1/x)\]
converges weakly to the measure $\sigma$ on $[0,\infty]$, and the
sequence of numbers \[
\gamma_{n}=-\log c_{n}+\sum_{j=1}^{k_{n}}\left[\int_{(0,\infty]}\frac{x^{2}-1}{x^{2}+1}\, d\mu_{nj}^{\circ}(1/x)-\log b_{nj}\right]\]
converges to the number $\gamma$ as $n\rightarrow\infty$. Moreover, as seen in \cite{BP2000, CG2008}, under the condition $\sigma(\{0\})=0=\sigma(\{\infty\})$, the above convergence criteria are known to be equivalent to the validity of the weak convergence \[\delta_{c_{n}}\circledast\mu_{n1}\circledast\mu_{n2}\circledast\cdots\circledast\mu_{nk_{n}} \Rightarrow \mu_{\circledast}^{\gamma,\sigma},\] where $\circledast$ denotes the classical multiplicative convolution arising from the product of commuting positive random variables and the law $\mu_{\circledast}^{\gamma,\sigma}$ is the $\circledast$-infinitely divisible law characterized by the classical L\'{e}vy-Hin\v{c}in's formula.

\section{local limit laws and convergence to free Brownian motions}
\subsection{Local limit theorems on $\mathbb{R}_+$}
Let $\nu$ be a non-degenerate $\boxtimes$-infinitely divisible law on $\mathbb{R}_{+}$, and let $\{\mu_{n}\}_{n=1}^{\infty} \subset \mathcal{M}_{+}$ and $\{k_{n}\}_{n=1}^{\infty} \subset \mathbb{N}$ be two sequences such that $k_n\rightarrow\infty$ and the measures $\nu_{n}=(\mu_n)^{\boxtimes k_n}$ converge weakly to $\nu$ as $n \rightarrow \infty$.

Recall from \cite{BB2005} that the $\eta$-transform $\eta_{\nu}$ extends continuously to $\overline{\mathbb{C}^{+}}\setminus \{0\}$ and the extension takes the half-line $(0,\infty)$ to a continuous curve in $\mathbb{C}^{+}\cup \mathbb{R}$. Moreover, $\eta_{\nu}$ is analytic at points $x\in (0,\infty)$ where $\Im{\eta_{\nu}(x)}>0$. The density $f_{\nu}$ of the limit law $\nu$ is continuous except on the finite set $\{x>0:\eta_{\nu}(1/x)=1\} \cup \{0\}$, and the function $f_{\nu}$ extends analytically to all points where it is different from zero. Here, and in the sequel, the extensions of relevant $\eta$-transforms and density functions will be denoted by the same notations if they should exist.

The theorem below shows that the regularity properties of $\nu$ get passed to the attracted measures $\nu_{n}$ in the limit process.

\begin{thm}\label{thm:4.1}
Let $I \subset (0,\infty)$ be a compact interval such that the density $f_{\nu}$ is continuous and $\min f_{\nu}>0$ on $I$. Then we have:
\begin{enumerate}[$(1)$]
\item For sufficiently large $n$, the measure $\nu_n$ becomes Lebesgue absolutely continuous on $I$, and its density function $f_{n}=d\nu_n/dx$ is continuous on $I$.
\item There exists a neighborhood $U$ of the interval $I$ such that the density $f_{n}$ extends analytically to $U$ for large $n$, and these extensions converge locally uniformly in $U$ to the analytic extension of $f_{\nu}$ as $n \rightarrow \infty$.
\end{enumerate}
\end{thm}

\begin{proof}

Let $J=\{x:1/x\in I\}$. By the relationship
\[\int_{\mathbb{R}_{+}}\frac{1}{z-x}\, d\nu(x) = \frac{1}{z[1-\eta_{\nu}(1/z)]}, \quad z \in \Omega,\]
and the Stieltjes inversion formula, the density $f_{\nu}$ is the $dx$-a.e. limit of \[-\frac{1}{\pi}\Im \frac{1}{(x+iy)[1-\eta_{\nu}(1/(x+iy))]}\quad (y \rightarrow 0^{+}).\] Thus, our hypothesis on the density
$f_{\nu}$ implies that the image $\eta_{\nu}\left(J\right)$ is a continuous
and bounded curve in $\mathbb{C}^{+}$, and the transform $\eta_{\nu}$ extends analytically to a bounded
neighborhood $V_1$ of the interval $J$ in the right half-plane. By choosing a smaller neighborhood if needed, we may assume that $V_2=\eta_{\nu}\left(V_{1}\right)$ is a neighborhood of $\eta_{\nu}\left(J\right)$ that lies in
$\mathbb{C}^{+}$.

Next, we consider the analytic map $H_{\nu}:\Omega \rightarrow\mathbb{C}$ given by $H_{\nu}(z)=z\Sigma_{\nu}(z)$, where the function $\Sigma_{\nu}$
is defined as in the free L\'{e}vy-Hin\v{c}in's formula (2.2) for the
infinitely divisible law $\nu$. By analytic continuation, we have  \begin{equation}\label{eq:4.1}
       H_{\nu}\left(\eta_{\nu}(z)\right)=z,\quad z\in V_{1}.
    \end{equation} By subordination, for each $n\geq 1$, there exists an analytic function $\omega_{n}:\Omega \rightarrow \Omega$ such that
   \begin{equation}\label{eq:4.2}
       \eta_{\nu_n}(z)=\eta_{\mu_n}\left(\omega_n(z) \right), \quad z\in\Omega.
   \end{equation}Moreover, an inversion relationship holds:
    \begin{equation}\nonumber
       H_n(\omega_n(z))=z,\quad z\in\Omega,
    \end{equation}where the inverse $H_{n}$ is given by the formula \[H_n(z)=z[B_{\mu_n}(z)]^{k_n-1}.\] We note for further reference that the subordination function $\omega_{n}$ can be calculated explicitly as follows. By free and Boolean L\'{e}vy-Hin\v{c}in's formulas, the function $B_{\mu_{n}}$ is the $\Sigma$-transform of a $\boxtimes$-infinitely divisible law $\rho_n$. Note that the function $H_n(z) = z\Sigma_{\tau_{n}}(z)$ where the law $\tau_n = (\rho_n)^{\boxtimes (k_n-1)}$ is also $\boxtimes$-infinitely divisible. Hence, by analytic continuation, we obtain that
     \begin{equation}\label{eq:4.3}
       \eta_{\tau_n}\left(H_{n}(z)\right)=z \end{equation} for any $z\in \Omega$ with $H_{n}(z) \in \Omega$. In particular, the function $\omega_{n}$ must be equal to the $\eta$-transform of the freely infinitely divisible law $\tau_{n}$.

Now, Proposition 2.2 and the weak convergence $\nu_n \Rightarrow \nu$ show that the measures $\mu_n$ converge weakly to the point mass $\delta_{1}$ as $n\rightarrow \infty$, and hence the sequence $\{\mu_n\}_{n=1}^{\infty}$ forms an infinitesimal array with identical rows. Consequently, Theorem 3.2 implies that
       \begin{equation}\nonumber
          H_n(z)=z[B_{\mu_n}(z)]^{k_{n}-1}\rightarrow z\Sigma_{\nu}(z)=H_{\nu}(z)
       \end{equation}
locally uniformly in $V_{2}$ as $n\rightarrow \infty$. By virtue of (\ref{eq:4.1}), this means that \[h_{n}(z)=H_{n}\left(\eta_{\nu}(z)\right) \rightarrow z \quad (n \rightarrow \infty)\] locally uniformly for $z\in V_1$. Then, by the argument principle, each function $h_{n}$ will be invertible when $n$ is sufficiently large. Also, since $h_{n}(z)\approx z$ locally uniformly in $V_1$ for such $n$, the image of $V_1$ under the map $h_n$ cannot deviate too much from $V_1$ itself. Hence, we can find a neighborhood $V_{3}$ such that $J \subset V_3 \subset V_1$ and the map $h_n$ has an analytic inverse $h_n^{-1}:V_{3} \rightarrow V_1$ so that \[H_{n}\left(\left(\eta_{\nu}\circ h_n^{-1}(z)\right)\right)=z,\quad z\in V_3.\] This fact and the inversion relationship (\ref{eq:4.3}) imply that the analytic function $\eta_{\nu}\circ h_n^{-1}$ coincides with the $\eta$-transform of the law $\tau_n$ in $V_3\cap \mathbb{C}^{+}$ for large $n$. Therefore, we can take $\eta_{\nu}\circ h_n^{-1}$ as an analytic continuation for the subordination function $\omega_n$ in the neighborhood $V_3$. Notice that we also have $\eta_{\nu}\circ h_n^{-1}\left(V_3\right) \subset V_2 \subset \mathbb{C}^{+}$ by our construction. In addition, \eqref{eq:4.2} shows that the equation $H_n(\omega_n(z))=z$ can be rewritten as \[\eta_{\nu_n}(z)=z\left[\frac{\omega_{n}(z)}{z}\right]^{k_n/(k_n-1)}, \quad z\in V_3,\] showing that the $\eta$-transform $\eta_{\nu_n}$ also extends analytically to $V_3$ for large $n$. (Here the principal branch of the power function is used.) Meanwhile, the subordination function $\omega_n$ converge to $\eta_{\nu}$ locally uniformly in $V_3$ because $h_{n}^{-1}(z) \rightarrow z$ locally uniformly for $z \in V_3$. Consequently, the $\eta$-transforms $\eta_{\nu_n}$ also converge to the same limit $\eta_{\nu}$ locally uniformly in $V_3$ by the subordination formula \eqref{eq:4.2} and the weak convergence $\mu_n \Rightarrow \delta_1$.

Finally, the desired neighborhood $U$ is given by
\[U=\{1/z:\, z\in V_3 \},\] and the results concerning the absolute continuity of the laws $\nu_n$ and the convergence properties of their densities follow at once from the method of Stieltjes inversion.
\end{proof}

We remark that the neighborhood $U$ in the preceding theorem can be chosen to be symmetric with respect to the half-line $(0,\infty)$, and then the analytic extension of the density $f_{n}$ can be taken as: \[f_{n}(z) = \frac{1}{2\pi i}\left[\frac{1}{z[1-\overline{\eta_{\nu_{n}}(1/\overline{z})}]} - \frac{1}{z[1-\eta_{\nu_{n}}(1/z)]}\right], \quad z\in U.\]

Recall that the free multiplicative Brownian motion on $\mathbb{R}_{+}$ is a multiplicative free L\'{e}vy process whose marginal laws $\{\chi_t\}_{t>0}$ are characterized as follows: \[\Sigma_{\chi_t}(z)= \exp\left(\frac{t+tz}{2z-2}\right), \quad z\in \Omega.\]

The following result from \cite{BianeJFA, Zhong1} gives a detailed description of the law $\chi_t$, showing that each $\chi_t$ belongs to the class $\mathcal{M}_{0}$ in the original definition of superconvergence \cite{BV1995sup}. For any fixed $t>0$, the law $\chi_t$ is Lebesgue absolutely continuous with a continuous density $f_t$ supported on the interval $[x_1,x_2] \subset (0,\infty)$, where \[ x_1=x_1(t)=\left[1+t/2-\sqrt{t(t+4)}/2\right]\exp\left(-\sqrt{t(t+4)}/2\right)\] and $x_2=x_2(t)=1/x_1(t)$. We also have that $f_t(x_1)=0=f_t(x_2)$ and for any $x\in (x_1,x_2)$, the value $f_t(x)$ is given by $\Im z / (2\pi x)$, where $z$ is the unique solution with a positive imaginary part of the equation \[z\exp(t(z-1/2))=x(z-1).\]

Thus, Theorem 4.1 yields a multiplicative analogue of the superconvergence phenomenon:

\begin{cor}
Let $\varepsilon>0$ be a given small number, and suppose that the sequence $\nu_n = (\mu_n)^{\boxtimes k_n}$ converges to $\chi_t$ weakly
as $n \rightarrow \infty$.
Then the density $f_n=d\nu_n/dx$ is analytic and strictly positive in
$[x_1(t)+\varepsilon,x_2(t)-\varepsilon]$
for sufficiently large $n$, and the $k$-th derivative $f_{n}^{(k)}$ converges to the $k$-th derivative $f_t^{(k)}$ uniformly on
$[x_1+\varepsilon,x_2-\varepsilon]$ for any integer $k\geq0$.
\end{cor}

\subsection{Local limit theorems on $\mathbb{T}$}
Let $\{\mu_{n}\}_{n=1}^{\infty}$ be measures in $\mathcal{M}_{\mathbb{T}}$ such that the sequence $\nu_n=(\mu_{n})^{\boxtimes k_n}$ tends weakly to a non-degenerate law $\nu \in \mathcal{M}_{\mathbb{T}}$ as $n\rightarrow \infty$,
where $k_n \geq 2$ are integers such that $k_n\rightarrow \infty$. As in the positive half-line case, it follows that $\mu_n \Rightarrow \delta_{1}$, and hence the sequence $\{\mu_{n}\}_{n=1}^{\infty}$ is infinitesimal.

The limit law $\nu$ is again $\boxtimes$-infinitely divisible \cite{BB2008a}, and all the regularity properties mentioned in Section 4.1 also hold for the measure $\nu$ in this case \cite{BB2005}. Thus, the law $\nu$ has at most one atom and has no singular continuous component relative to the arc-length measure $d\theta$ on $\mathbb{T}$. The $\eta$-transform $\eta_{\nu}$ extends continuously to $\mathbb{T}$, and it is analytic at points $\xi \in \mathbb{T}$ such that $\left|\eta_{\nu}(\xi)\right|<1$. The density $f_\nu= d\nu/d\theta$ is continuous everywhere except at the point $\xi \in \mathbb{T}$ where $\eta_{\nu}(\xi)=1$, and $f_\nu$ is analytic at all points where it is positive.

The analogue of Theorem 4.1 in the current setting is the following
\begin{thm}\label{thm:4.3}
Let $I \subseteq \mathbb{T}$ be a connected arc on which the density $f_{\nu}$ is continuous and bounded away from zero. Then there exists a neighborhood $U$ of the arc $I$ such that for sufficiently large $n$, the measure $\nu_n$ is absolutely continuous on $I$ with respect to $d\theta$ and its density $f_n$ extends analytically to $U$. Moreover, the analytic extension of $f_n$ converges locally uniformly in $U$ to that of the limiting density $f_{\nu}$.
\end{thm}
\begin{proof}
We proceed as in the proof of Theorem \ref{thm:4.1}. Setting $J=\{1/\xi: \xi \in I\}$ as before, the formula
   \begin{equation}\label{eq:4.4}
     \frac{1}{2\pi}\frac{1-|\et{\nu}(z)|^2}{|1-\et{\nu}(z)|^2} = \frac{1}{2\pi}\int_{\mathbb{T}}\Re{\left[\frac{\xi + z}{\xi - z}\right]}\, d\nu\left(1/\xi\right), \quad z\in \mathbb{D},
   \end{equation}
and the general theory of Poisson integrals imply that the density of the law $d\nu\left(1/\xi\right)$ is determined by the boundary radial limits of \eqref{eq:4.4}. Since the change of variable $\xi \mapsto 1/\xi$ preserves the arc-length measure $d\theta$, we deduce from the assumptions on $f_{\nu}$ that the image of $J$ under the map $\eta_{\nu}$ is a continuous curve which lies in $\mathbb{D}$ and is bounded away from the circle $\mathbb{T}$. In addition, the map $\eta_{\nu}$ admits an analytic continuation to a neighborhood $V_1$ of the arc $J$ such that the domain $V_2=\eta_{\nu}\left(V_1\right)$ is contained in $\mathbb{D}$.

In this case, it is still true that the $\eta$-transform of the free convolution $\nu_{n}$ will be subordinated to the function
$\eta_{\mu_n}$; however, the underlying probability measure associated with this subordination function may not be $\boxtimes$-infinitely divisible. In what follows, we will use the global inversion result to find a slightly different subordination property which will give us the infinite divisibility. This method was first used in \cite{BB2005} for constructing fractional free convolution powers, where the use of infinite divisibility is implicit in that construction. We begin by introducing the auxiliary free convolution \[\rho_n=\mu_n\boxtimes \mu_n, \quad n\geq 1.\] Then, by Proposition 3.3 (1) in \cite{BB2005}, every such a measure $\rho_n$ will be $\bmultg$-infinitely divisible and hence its $B$-transform has a square root \[[B_{\rho_n}(z)]^{1/2}=\exp\left(u_n(z)\right), \quad z\in \mathbb{D},\] where $\Re u_n(z) \geq 0$ for $z \in \mathbb{D}$ and $[B_{\rho_n}(0)]^{1/2}=1/\text{m}(\mu_n)$. In particular, this property allows us to define the following analytic function \begin{equation}\label{eq:4.5}\Phi_{n}(z)=z[B_{\rho_{n}}(z)]^{(k_n-2)/2}=z\exp\left((k_n-2)u_n(z)\right), \quad z \in \mathbb{D} ,\end{equation} where the $(k_n-2)/2$ power is chosen to be $1/\left[\text{m}(\mu_n)\right]^{k_n-2}$ for $z=0$. Clearly, we also have \[\Phi_{n}(0)=0 \quad \text{and} \quad \left|\Phi_{n}(z)\right| \geq \left|z\right|\] for all $z \in \mathbb{D}$.

Now, by Proposition 2.3, the function $\Phi_n$ has a unique global inverse $\omega_n:\overline{\mathbb{D}} \rightarrow \overline{\mathbb{D}}$ such that the inversion relationship \[\Phi_{n} \left(\omega_{n}(z)\right)=z\] holds in the open disk $\mathbb{D}$. Since $\left|\omega_{n}(z)\right| \leq \left|z\right|$ for $z\in \mathbb{D}$, the composition $\eta_{\rho_n} \circ \omega_n$ in $\mathbb{D}$ is the $\eta$-transform of a law $\beta_n \in \mathcal{M}_{\mathbb{T}}$. We shall verify that this measure $\beta_n$ is exactly the free convolution $\nu_n=(\mu_{n})^{\boxtimes k_n}$, and hence the desired analytic subordination \begin{equation}\label{eq:4.6}\eta_{\nu_n}(z)=\eta_{\rho_n}\left(\omega_n(z)\right), \quad z\in \mathbb{D},\end{equation} will be established. To this end, note that \begin{equation}\nonumber
        \begin{split}
         \Sigma_{\beta_n}(z)&=\frac{1}{z}\,\eta_{\beta_n}^{-1}(z)=\frac{1}{z}(\omega_{n}^{-1}\circ \eta_{\rho_n}^{-1})(z)=\frac{1}{z}(\Phi_{n} \circ \eta_{\rho_n}^{-1})(z)\\
                      &=\frac{1}{z}\,\eta_{\rho_n}^{-1}(z)[(B_{\rho_{n}}\circ \eta_{\rho_n}^{-1})(z)]^{(k_n-2)/2}\\ &=\Sigma_{\rho_n}(z)[\Sigma_{\rho_n}(z)]^{(k_n-2)/2}=[\Sigma_{\mu_n}(z)]^{k_n}=\Sigma_{\nu_n}(z)
                             \end{split}
      \end{equation} for $z$ near the origin. Since the $\Sigma$-transform of a measure determines the underlying measure uniquely, we obtain $\beta_n=\nu_n$ as claimed.

Furthermore, the Herglotz integral formula of the function $u_{n}$ in \eqref{eq:4.5} and the free L\'{e}vy-Hin\v{c}in's formula \eqref{eq:2.4} imply that $\Phi_{n}(z)=z\Sigma_{\tau_n}(z)$, where the measure $\tau_{n}$ is a $\boxtimes$-infinitely divisible law in $\mathcal{M}_{\mathbb{T}}$. Therefore the subordination function $\omega_n$ is in fact the $\eta$-transform of the law $\tau_n$, and we obtain yet another inversion formula which is the analogue of (\ref{eq:4.3}) in the current setting: \[\eta_{\tau_n}\left(\Phi_{n}(z)\right)=z\] for any $z\in \mathbb{D}$ such that $\Phi_{n}(z) \in \mathbb{D}$.

With these subordination and inversion formulas constructed, the rest of the proof now follows the same lines as in the proof of Theorem 4.1. We conclude that for large $n$, the functions $\omega_n$ and $\eta_{\nu_n}$ both extend analytically to a neighborhood $V_3$ of the arc $J$, and their extensions tend to the same limit $\eta_{\nu}$ locally uniformly in $V_3$. The proof is then finished by taking the boundary values of the integral formula \eqref{eq:4.4}. We omit this rather similar argument and only mention that the local uniform convergence of $\{\Phi_n\}_{n=1}^{\infty}$ to the function $z\Sigma_{\nu}(z)$ in the neighborhood $V_2$ is guaranteed by Theorem 3.5 from \cite{Wang}, which is the Gnedenko type convergence criterion for the weak convergence $\nu_n \Rightarrow \nu$.
\end{proof}

We next consider the special case where the limit law is the Haar measure on $\mathbb{T}$. In this case, we define the free convolution $\rho_{n}=\mu_{n} \boxtimes \mu_{n}$ as before and recall the corresponding subordination function
$\omega_n(z)=\eta_{\tau_n}(z)$ and its left inverse $\Phi_{n}=z\Sigma_{\tau_n}(z)$ for $z\in \mathbb{D}$. Here the subordination distribution $\tau_n$ is both $\boxtimes$- and $\bmultg$-infinitely divisible, see Section 2.2.

By Proposition 2.3, $\omega_{n}$ is an injective and continuous function satisfying \[\omega_{n}\left(\mathbb{D}\right)=\{z\in \mathbb{D}:\left|\Phi_{n}(z)\right|<1\}\] and the boundary $\partial\omega_{n}\left(\mathbb{D}\right)=\omega_{n}\left(\mathbb{T}\right)$. Moreover, we have \[\omega_{n}\left(\mathbb{T}\right) \cap \mathbb{D} = \{z\in \mathbb{D}:\left|\Phi_{n}(z)\right|=1\},\] and if $\xi\in \omega_{n}\left(\overline{\mathbb{D}}\right) \cap \mathbb{T}$ then the entire radius $\{r\xi:0 \leq r <1\}$ is contained in $\omega_{n}\left(\mathbb{D}\right)$. Finally, if $\xi \in \mathbb{T}$ is such that $\omega_{n}(\xi) \in \mathbb{D}$, then $\omega_{n}$ has an analytic continuation at the point $\xi$.

We now prove the superconvergence to Haar measure $d\theta/2\pi$. Notice that the infinitesimality and weak convergence assumptions are not required in the following result.
\begin{thm}\label{thm:4.4}
Let $\{\mu_n\}_{n=1}^{\infty}$ be a sequence of measures in $\MT$ such that \[\lim_{n \rightarrow \infty}\left[\text{m}(\mu_n)\right]^{k_n-2}=0.\] Set $\nu_n=(\mu_n)^{\boxtimes k_n}$. Then for sufficiently large $n$, we have:
\begin{enumerate}[$(1)$]
\item Each measure $\nu_{n}$ is absolutely continuous with the support $\mathbb{T}$.
\item The density function $f_{n}=d\nu_{n}/d\theta$ is continuous on $\mathbb{T}$, and $f_n$ extends analytically to every point on $\mathbb{T}$.
\item The densities $f_n$ converge to $1/2\pi$ uniformly on $\mathbb{T}$ as $n \rightarrow \infty$.
\end{enumerate}
\end{thm}
\begin{proof}
We first write each function $\Phi_{n}$ in its free L\'{e}vy-Hin\v{c}in's representation
\begin{equation}\nonumber
              \Phi_n(z)=\alpha_{n}\,z\,\exp\left(\int_{\mathbb{T}}\frac{1+\xi z}{1-\xi z}\,d\sigma_n(\xi) \right), \quad z\in \mathbb{D},
            \end{equation} and observe from it that the derivative \[\Phi_n^{\prime}(0)=\lim_{z \rightarrow 0}\Phi_{n}(z)/z=\alpha_ne^{\sigma_n\left(\mathbb{T}\right)}.\] On the other hand, the formula (\ref{eq:4.5}) shows that the same derivative can also be calculated as follows: \[ \Phi_n^{\prime}(0)=\left[B_{\rho_{n}}(0)\right]^{(k_n-2)/2}=1/\left[\text{m}(\mu_n)\right]^{k_n-2}.\] Therefore, taking the modulus of the last calculation yields \[\lim_{n\rightarrow \infty}\sigma_n\left(\mathbb{T}\right)=\infty.\]

Next, we prove that the images $\omega_{n}\left(\overline{\mathbb{D}}\right)$ are all contained entirely in the open disc $\mathbb{D}$ when $n$ is sufficiently large. Indeed, if this were not true, then the set $\omega_{n}\left(\overline{\mathbb{D}}\right)$ would intersect the boundary $\mathbb{T}$ frequently as $n\rightarrow \infty$. In particular, there must be a sequence $\xi_n$ such that $\xi_n \in \omega_{n}\left(\overline{\mathbb{D}}\right) \cap \mathbb{T}$ for infinitely many $n$'s. Then the mapping properties of $\omega_n$ would imply that the relationship $z_n=\xi_n/2 \in \omega_{n}\left(\mathbb{D}\right)$ and the inequality $\left|\Phi_{n}(z_n)\right|<1$ both occur infinitely often. This however is a contradiction, because \begin{equation}\nonumber
        \begin{split}
         2>2\left|\Phi_{n}(z_n)\right|&=\exp\left(\int_{\mathbb{T}}\Re{\left[\frac{1+\xi z_n}{1-\xi z_n}\right]}\,d\sigma_n(\xi) \right)\\
                      &\geq \exp \left(\frac{1-\left|z_n\right|}{1+\left|z_n\right|} \, \sigma_n \left(\mathbb{T}\right)\right)=e^{\sigma_n\left(\mathbb{T}\right)/3}
                             \end{split}
      \end{equation} and $\lim_{n\rightarrow \infty}\sigma_n\left(\mathbb{T}\right)=\infty$.

      Now, by choosing a $N>0$ so that $\omega_{n}\left(\overline{\mathbb{D}}\right) \subset \mathbb{D}$ for $n\geq N$, the subordination formula (\ref{eq:4.6}) shows that the $\eta$-transform of the law $\nu_n$ extends continuously to $\mathbb{T}$ for such $n$ and this extension takes values in the open disc $\mathbb{D}$. Consequently, the Poisson integral formula (\ref{eq:4.4}) implies that the density function $f_n$ is continuous and positive everywhere on its support $\mathbb{T}$. Finally, the analyticity property of $f_n$ comes from that of the map $\omega_n$ and from the analytic subordination \eqref{eq:4.6}.

The preceding argument proves the statements (1) and (2). To prove (3), let us introduce the auxiliary function \[g(r)=\frac{(1+r) \log r}{1-r}, \quad 0<r<1,\] and observe that $g$ is a strictly increasing and continuous function on the interval $(0,1)$ such that $g(1^{-})=-2$ and $g(0^{+})=-\infty$. In particular, the function $g$ has an inverse $g^{-1}$ such that \[\lim_{x\rightarrow -\infty}g^{-1}(x)=0.\] For $n\geq N$, we calculate \begin{eqnarray*}
\omega_n\left(\overline{\mathbb{D}}\right) & = & \{z \in \mathbb{D}:\left|\Phi_n(z)\right|\leq 1\}\\& \subset & \{z \in \mathbb{D}:g\left(\left| z \right|\right)\leq - \sigma_n \left(\mathbb{T} \right)\}\\& = & \{z \in \mathbb{D}:\left| z \right| \leq g^{-1}\left(- \sigma_n \left(\mathbb{T} \right)\right)\}.
\end{eqnarray*}
Since the last disc shrinks uniformly to the origin as $n\rightarrow \infty$, we deduce from the subordination formula (\ref{eq:4.6}) that the $\eta$-transforms $\eta_{\nu_{n}}$ converge to zero uniformly on $\overline{\mathbb{D}}$, and in particular, uniformly on the circle $\mathbb{T}$. Therefore, (3) follows from the integral formula (\ref{eq:4.4}). \end{proof}

We remark that Theorem 4.4 (3) does imply the weak convergence $\nu_{n} \Rightarrow d\theta/2\pi$.

\subsection{Superconvergence to free unitary Brownian motion}
Denote by $\lambda_t$ the law of free unitary Brownian motion at time $t>0$. The $\Sigma$-transform of $\lambda_t$ is given by
   \begin{equation}\nonumber
      \Sigma_{\lambda_t}(z)=\exp\left(\frac{t+tz}{2-2z}\right), \quad z \in \mathbb{D}.
   \end{equation}
From Biane's work \cite{BianeJFA}, we know that the infinitely divisible law $\lambda_t$ is absolutely continuous with respect to $d\theta$ on $\mathbb{T}$, and its support $\text{supp}\left(\lambda_t\right)$ is the entire circle $\mathbb{T}$ when $t>4$. For $0<t\leq 4$, we have \[\text{supp}\left(\lambda_t\right) = \left\{e^{i\theta}: -\theta_1 \leq \theta \leq \theta_1\right\},\] where $\theta_1=\theta_0+\sqrt{t-t^{2}/4} \in (0,\pi]$ and $\theta_0=\arccos\left(1-t/2\right)$. The density $d\lambda_t/d\theta$ is continuous everywhere on $\mathbb{T}$ and is strictly positive on the interior of $\text{supp}\left(\lambda_t\right)$, except at $-1$ when $t=4$. The value of this density at a point $\xi \in \mathbb{T}$ is given by $\Re z / 2\pi$, where $z$ is the only solution with a positive real part of the equation \[(z-1)\exp\left(tz/2\right)=(z+1)\xi.\]

For $t\in(0,4]$, the image $\eta_{\lambda_t}\left(\mathbb{T}\right)$ is a continuous simple closed curve whose parametrization is described as follows. First, the curve $\eta_{\lambda_t}\left(\mathbb{T}\right)$ is symmetric about the $x$-axis. Second, the interior of $\text{supp}\left(\lambda_t\right)$ is mapped bijectively to a simple curve in the open disc $\mathbb{D}$ by the function $\eta_{\lambda_t}$, and if we write the polar coordinate $\eta_{\lambda_t}\left(e^{i\theta}\right)=re^{i\psi}$ for $\theta \in (-\theta_1,\theta_1)$ then the functions $r=r(\theta)$ and $\psi=\psi(\theta)$ are defined implicitly through the equation
\begin{equation}\label{eq:4.7}
         t\,(r^2-1)=2\log r\,(1-2r\cos\psi+r^2). \end{equation} Moreover, the modulus function $r(\theta)$ is a continuous function which increases to $1$ as $\theta \rightarrow \theta_1$, and the polar angle $\psi(\theta)$ is also a continuous function with \[-\theta_0 < \psi (\theta)< \theta_0, \quad \theta \in (-\theta_1,\theta_1),\] and $\psi(\theta) \rightarrow \theta_0$ as $\theta \rightarrow \theta_1$. Finally, the function $\eta_{\lambda_t}$ maps the complementary arc $\{e^{i\theta}: \theta_1 \leq \left|\theta\right| \leq \pi\}$ onto the arc $\{e^{i\theta}: \theta_0 \leq \left|\theta\right| \leq \pi\}$ such that $\eta_{\lambda_t}\left(e^{i\theta_1}\right)=e^{i\theta_0}$, $\eta_{\lambda_t}\left(e^{-i\theta_1}\right)=e^{-i\theta_0}$, and $\eta_{\lambda_t}(-1)=-1$. We refer to \cite{Zhong1, Zhong2} for the details of these calculations.

Fix $t>0$. We shall study the superconvergence to the law $\lambda_t$. As usual, let $\{\mu_n\}_{n=1}^{\infty}$ be a sequence in $\mathcal{M}_{\mathbb{T}}$ and $k_n$ be positive integers tending to infinity. Suppose that the weak convergence $\nu_{n}=(\mu_n)^{\boxtimes k_n} \Rightarrow \lambda_t$ holds. Following Section 4.2, we have the subordination $\eta_{\nu_{n}}=\eta_{\rho_n} \circ \omega_n$ in $\mathbb{D}$, where the law $\rho_n=\mu_n \boxtimes \mu_n$ and the subordination map $\omega_n=\eta_{\tau_n}$ for some $\boxtimes$-infinitely divisible law $\tau_n$. The left inverse of $\omega_n$ is the function $\Phi_n(z)=z\Sigma_{\tau_n}(z)$, and Theorem 3.5 in \cite{Wang} implies that the limit \begin{equation}\label{eq:4.8}
         \lim_{n \rightarrow \infty}\Phi_n(z)=z\Sigma_{\lambda_t}(z)=z\exp\left(\frac{t+tz}{2-2z}\right) \end{equation} holds uniformly for $z$ in any compact subset of the disc $\mathbb{D}$.

Since the law $\tau_n$ is also $\bmultg$-infinitely divisible, the origin is a simple zero for the map $\omega_n$. So, the quotient \[g_n(z)=\frac{\omega_n(z)}{z}\]  is analytic in $\mathbb{D}$ and continuous on the closed disk $\overline{\mathbb{D}}$. Notice that the function $g_n(z) \neq 0$ for all $z \in \overline{\mathbb{D}}$, because $\omega_n$ is one to one in $\overline{\mathbb{D}}$. Our first lemma shows that for sufficiently large $n$, the range $g_n\left(\overline{\mathbb{D}}\right)$ will be bounded away from zero uniformly in $n$. \begin{lemma}\label{lemma:4.5}
There exist $N_1>0$ and a positive constant $C \in (0,1)$ such that \[\left|g_n(z)\right| \geq C, \quad n \geq N_1, \quad z \in \overline{\mathbb{D}}.\] \end{lemma}
\begin{proof}
Suppose, to the contrary, that there is a sequence $z_n \in \overline{\mathbb{D}}$ such that $\left|g_n(z_n)\right|\rightarrow 0$ as $n \rightarrow \infty$. By dropping to a subsequence if necessary, but without changing our notation, we assume that the sequence $z_n$ converges to a point $z\in \overline{\mathbb{D}}$. Note that \[\left|\omega_n(z_n)\right|=\left|z_n\right|\left|g_n(z_n)\right| \rightarrow \left|z\right|\cdot 0=0\] as $n \rightarrow \infty$. So, we have $\left|\omega_n(z_n)\right| \leq 1/2$ when $n$ is large enough. It follows from Proposition 2.3 that the function $\omega_n$ is analytic at such $z_n$, and therefore we have $z_n=\Phi_n\left(\omega_n(z_n)\right)$ by analytic continuation. Denoting $w_n=\omega_n(z_n)$, the equation \eqref{eq:4.8} shows that \[\left|z_n\right| \leq \left|\Phi_n(w_n) - w_n\Sigma_{\lambda_t}(w_n)\right| + \left|w_n\Sigma_{\lambda_t}(w_n)\right| \rightarrow 0\] as $n \rightarrow \infty$. Hence, we obtain $z=0$ and all $z_n$ will eventually fall into a compact neighborhood of the origin. Notice that the limit \eqref{eq:4.8} also yields the weak convergence $\tau_n \Rightarrow \lambda_t$, and as a result, the sequence $g_n$ converges locally uniformly in $\mathbb{D}$ to the function $\eta_{\lambda_t}(z)/z$. Therefore, we have \[0=\lim_{n \rightarrow \infty}g_n(z_n)=\eta_{\lambda_t}^{\prime}(0)=1/\Sigma_{\lambda_t}(0)=e^{-t/2},\] which is a contradiction. Therefore, the range of $g_n$ must be uniformly bounded away from zero when $n$ is large enough.
\end{proof}

Since each $g_n$ never vanishes in the open disc $\mathbb{D}$, $g_n$ has an analytic logarithm defined in $\mathbb{D}$. But, of course, the branch of this logarithm may vary as the index $n$ does. The point of our next lemma is that one can always take the principal logarithm of $g_n$, even at the boundary $\mathbb{T}$, when $n$ is large enough.
\begin{lemma}\label{lemma 4.6}
There exists $N_2>N_1$ such that \[g_n\left(\overline{\mathbb{D}}\right) \subset \{re^{i\theta}: C \leq r \leq 1, \left|\theta\right| \leq 2\}, \quad n \geq N_2.\] Consequently, the range $g_n\left( \overline{\mathbb{D}}\right)$ misses the interval $[-1,0]$ for all $n \geq N_2$, and therefore the principal logarithm $\log g_n$ is defined in $\overline{\mathbb{D}}$ for such $n$.
\end{lemma}
\begin{proof}
Since the image $g_n\left( \overline{\mathbb{D}}\right)$ is the closure of the set $g_n\left(\mathbb{D}\right)$ and $0 \notin g_n\left( \overline{\mathbb{D}}\right)$ for $n \geq N_1$, it suffices to show that the principal argument of any point in $g_n\left(\mathbb{D}\right)$ will never exceed $2$ in magnitude for large $n$.

To this end, let us first show that this phenomenon does happen for the functions $\eta_{\lambda_s}(z)/z$, $s>0$. Since $\eta_{\lambda_s}(\overline{z})=\overline{\eta_{\lambda_s}(z)}$ for $z \in \mathbb{D}$, it is enough to consider the image points of the form: $\eta_{\lambda_s}(z)=re^{i\theta}$ where the polar coordinates $0<r<1$ and $0 \leq \theta \leq \pi$. For such points, we have \begin{equation}\label{eq:4.9}
         \frac{\eta_{\lambda_s}(z)}{z}=\exp\left(\frac{s+s\eta_{\lambda_s}(z)}{2\eta_{\lambda_s}(z)-2}\right)=\exp\left(\frac{sr^2-s-2isr\sin \theta}{2-4r\cos\theta+2r^2}\right). \end{equation} Meanwhile, Proposition 2.3 (2) shows that \begin{equation}\label{eq:4.10}
        1>\left|re^{i\theta}\Sigma_{\lambda_s}\left(re^{i\theta}\right)\right|=r\exp\left(\frac{s-sr^2}{2-4r\cos\theta+2r^2}\right). \end{equation} It follows from \eqref{eq:4.9} and \eqref{eq:4.10} that the real part of $\eta_{\lambda_s}(z)/z$ is positive and the principal argument \begin{equation}\label{eq:4.11}
             0 \leq \arg\frac{\eta_{\lambda_s}(z)}{z} \leq \frac{2r\ln r}{r^2-1} \leq 1,
            \end{equation} as desired. Notice that this estimate is uniform for all $s>0$.

        We next estimate the argument of $g_n$. First, we recall the free L\'{e}vy-Hin\v{c}in's formula for the function $\Phi_n$:
\begin{equation}\nonumber
              \Phi_n(z)=\alpha_{n}\,z\,\exp\left(\int_{\mathbb{T}}\frac{1+\xi z}{1-\xi z}\,d\sigma_n(\xi) \right), \quad z\in \mathbb{D},
            \end{equation} and note that the convergence \eqref{eq:4.8} implies \[\lim_{n \rightarrow \infty} \alpha_n=1 \quad \text{and} \quad \sigma_n \Rightarrow (t/2)\,\delta_1.\] (cf. Theorem 3.5 (iii) of \cite{Wang}.) The second limit allows us to introduce a probability law $\varsigma_n \in \mathcal{M}_{\mathbb{T}}$ by the normalization $d\varsigma_n(\xi)=d\sigma_n(\xi)/\sigma_n\left(\mathbb{T}\right)$ for large $n$, say, for $n\geq M >N_1$. For any $z \in \mathbb{D}$ and $n \geq M$, note that \begin{equation}\nonumber
        \begin{split}
         \frac{1}{g_n(z)}=\frac{z}{\omega_n(z)}=\frac{\Phi_n\left(\omega_n(z)\right)}{\omega_n(z)}&=\alpha_{n}\exp\left(\frac{2\sigma_n \left(\mathbb{T}\right)}{2}\int_{\mathbb{T}}\frac{1+\xi \omega_n(z)}{1-\xi \omega_n(z)}\,d\varsigma_n(\xi) \right)\\
                      &=\alpha_{n}\exp\left(\frac{2\sigma_n \left(\mathbb{T}\right)}{2}\frac{1+\eta_{\varsigma_n} \left(\omega_n(z)\right)}{1-\eta_{\varsigma_n} \left(\omega_n(z)\right)}\right)\\ &=\alpha_{n}\,\Sigma_{\lambda_{2\sigma_n \left(\mathbb{T}\right)}}\left(\eta_{\varsigma_n} \left(\omega_n(z)\right)\right).
                             \end{split}
      \end{equation}

      On the other hand, let us examine the free convolution $\varsigma_n \boxtimes \lambda_{2\sigma_n \left(\mathbb{T}\right)}$. By analytic subordination, there are two analytic self-maps $h_n$ and $l_n$ in $\mathbb{D}$ such that $h_n(0)=0=l_n(0)$ and \[\eta_{\varsigma_n \boxtimes \lambda_{2\sigma_n \left(\mathbb{T}\right)}}(z)=\eta_{\varsigma_n}\left(h_n(z)\right)=\eta_{ \lambda_{2\sigma_n \left(\mathbb{T}\right)}}\left(l_n(z)\right), \quad z \in \mathbb{D}.\] Since any $\eta$-transform is a contraction in $\mathbb{D}$ and $\varsigma_n \in \mathcal{M}_{\mathbb{T}}$, we have \[h_n=\eta_{\varsigma_n}^{-1} \circ \eta_{\varsigma_n \boxtimes \lambda_{2\sigma_n \left(\mathbb{T}\right)}}\] in a small neighborhood of the origin. Then a formal calculation based on the multiplicative property of the $\Sigma$-transform leads to \begin{equation}\nonumber \begin{split}
         \alpha_{n}z&=\alpha_{n}\,h_{n}(z)\,\Sigma_{\lambda_{2\sigma_n \left(\mathbb{T}\right)}}\left(\eta_{\varsigma_n \boxtimes \lambda_{2\sigma_n \left(\mathbb{T}\right)}} (z)\right)\\
                      &=\alpha_{n}\,h_{n}(z)\exp\left(\sigma_n \left(\mathbb{T}\right)\frac{1+\eta_{\varsigma_n}\left(h_n(z)\right)}{1-\eta_{\varsigma_n} \left(h_n(z)\right)}\right)\\ &=\alpha_{n}\,h_n(z)\exp\left(\int_{\mathbb{T}}\frac{1+\xi h_n(z)}{1-\xi h_n(z)}\,d\sigma_n(\xi) \right) = \Phi_{n}\left(h_{n}(z)\right)
                             \end{split}
      \end{equation} for $z$ near the origin. Therefore, the uniqueness part of Proposition 2.3 implies that \[\omega_n(z)=h_n\left(\overline{\alpha_n}\,z\right)\] first in a small neighborhood of the origin, and then to the entire disc $\mathbb{D}$ by analytic continuation. This fact, together with the previous calculations regarding $1/g_n(z)$, gives the following representation for $g_n(z)$: \[g_n(z)=\overline{\alpha_{n}}\,\frac{\eta_{\lambda_{2\sigma_n \left(\mathbb{T}\right)}}\left(l_n\left(\overline{\alpha_n}\,z\right)\right)}{l_n\left(\overline{\alpha_n}\,z\right)},\quad z \in \mathbb{D}.\] We conclude from this formula and the uniform estimate \eqref{eq:4.11} that the image $g_n\left(\mathbb{D}\right)$ is always contained in the truncated circle sector \[A_n=\overline{\alpha_n}\,\cdot\{re^{i\theta}: C \leq r \leq 1, -1\leq \theta \leq 1\}.\]

      Finally, the desired $N_2$ can be found because $\alpha_n \rightarrow 1$ as $n \rightarrow \infty$.       \end{proof}

      We now present the superconvergence result.
      \begin{thm}\label{thm:4.8}
      Suppose that the free convolutions $\nu_n=(\mu_n)^{\boxtimes k_n}$ converge weakly to the law $\lambda_t$ as $n\rightarrow \infty$.
Then for sufficiently large $n$, we have:
\begin{enumerate}[$(1)$]
\item Each measure $\nu_n$ is absolutely continuous with respect to $d\theta$ on $\mathbb{T}$.
\item The density $f_{n}=d\nu_n/d\theta$ is continuous on $\mathbb{T}$, and $f_n$ extends analytically to each point in the interior of $\emph{supp}\left(\lambda_t\right)$.
\item The densities $f_n$ converge to $d\lambda_t/d\theta$ uniformly on $\mathbb{T}$ as $n \rightarrow \infty$.
\end{enumerate}
      \end{thm}
      \begin{proof} For $t>4$, the statements (1) to (3) follow directly from Theorem 4.3 because the continuous density $d\lambda_t/d\theta$ is positive everywhere on $\text{supp}\left(\lambda_t\right)=\mathbb{T}$. We shall focus on the case $t\in (0,4]$.

      Fix $t\in (0,4]$. Recall the definition of the angles \[\theta_0=\arccos\left(1-t/2\right) \quad \text{and} \quad \theta_1=\theta_0+\sqrt{t-t^{2}/4},\] and let $\varepsilon>0$ be any small number such that $\theta_0 < \theta_1 - \varepsilon < \theta_1$ and \[\varepsilon< \frac{1}{10}\min\{\text{d}[1,\eta_{\lambda_t}\left(\mathbb{T}\right)], \, \text{d}[0,\eta_{\lambda_t}\left(\mathbb{T}\right)]\}.\] Here the notation $\text{d}[z,\eta_{\lambda_t}\left(\mathbb{T}\right)]$ means the minimum distance between $z \in \mathbb{C}$ and the curve $\eta_{\lambda_t}\left(\mathbb{T}\right)$. (Notice that $0,1 \notin \eta_{\lambda_t}\left(\mathbb{T}\right)$.)
We introduce two positive quantities: \[\delta=\delta(\varepsilon)=\theta_0 - \arg \eta_{\lambda_t}\left(e^{i(\theta_1-\varepsilon)}\right), \quad d=d(\varepsilon)=1 - \left|\eta_{\lambda_t}\left(e^{i(\theta_1-\varepsilon)}\right)\right|,\] and define the symmetric open arc \[I_{\varepsilon}=\{e^{i\theta}:-\theta_1+\varepsilon < \theta < \theta_1-\varepsilon\}.\] Note that $I_{\varepsilon} \subset \text{supp}\left(\lambda_t\right)$. By the continuity of $\eta_{\lambda_t}$, we have \[\lim_{\varepsilon \rightarrow 0^{+}}\delta(\varepsilon)=0=\lim_{\varepsilon \rightarrow 0^{+}}d(\varepsilon).\] So, we may assume further that $d<1/4$.

      We know from Theorem 4.3 that the continuous curve $\omega_n\left(I_{\varepsilon}\right)$ lies in $\mathbb{D}$, and it is uniformly close to the curve $\eta_{\nu}\left(I_{\varepsilon}\right)$ as $n \rightarrow \infty$. In particular, we may select a $N_3>N_2$ such that  \begin{equation}\label{eq:4.12}
             \left|\omega_n(\xi)-\eta_{\lambda_t}(\xi)\right|< \varepsilon, \quad \xi \in I_{\varepsilon}, \quad n \geq N_3.
            \end{equation}

            We first prove that for sufficiently large $n$, the rest of the boundary satisfies \begin{equation}\label{eq:4.13}
             \omega_n\left(\mathbb{T} \setminus I_{\varepsilon}\right) \subset \{z\in \overline{\mathbb{D}}: 1-2d \leq \left|z\right| \leq 1, \theta_0 - \delta \leq \left|\arg z \right| \leq \pi \}.
            \end{equation} To this end, observe that for any $\xi \in \mathbb{T} \setminus I_{\varepsilon}$ and for sufficiently large $n$, the argument $\arg \omega_n(\xi)$ cannot lie between $-\theta_0 + \delta$ and $\theta_0 - \delta$. Indeed, if this is the case, then the radius \[\{r\omega_n(\xi): 0< r \leq 1/\left|\omega_n(\xi)\right|\}\] will intersect the continuous curve $\omega_n\left(I_{\varepsilon}\right)$ at some point $\xi^{\prime} \in I_{\varepsilon}$, which is not possible by the $\boxtimes$-infinite divisibility of the subordination distribution $\tau_n$ and Proposition 2.4. Hence, we always have \[\theta_0 - \delta \leq \left|\arg \omega_n(\xi) \right| \leq \pi, \quad \xi \in \mathbb{T} \setminus I_{\varepsilon},\] whenever $n$ is sufficiently large.

            Next, concerning the condition on the modulus, suppose to the contrary that the curve $\omega_n\left(\mathbb{T} \setminus I_{\varepsilon}\right)$ often deviates from $\mathbb{T}$ by the amount of more than $2d$ as $n \rightarrow \infty$. Then there exists a sequence $\xi_n \in \mathbb{T} \setminus I_{\varepsilon}$ such that $\left|\omega_n (\xi_n)\right|<1-2d$ for all $n \geq 1$. By dropping to convergent subsequences, we have that $\omega_n (\xi_n) \rightarrow z$ and $\xi_n \rightarrow \xi$, where the limit point $\xi \in \mathbb{T} \setminus I_{\varepsilon}$ and $\left|z\right| \leq 1-2d$. Hence the inversion relationship $\xi_n=\Phi_n\left(\omega_n(\xi_n)\right)$ and the convergence \eqref{eq:4.8} show that $\eta_{\lambda_t}(\xi)=z$. Since $z \in \mathbb{D}$, we must have \[ \theta_1 - \varepsilon \leq \left|\arg \xi \right| \leq \theta_1.\] So, the continuity of $\eta_{\lambda_t}$ and the parametrization equation \eqref{eq:4.7} imply that $\left|z\right|>1-d$, a contradiction to the inequality $\left|z\right| \leq 1-2d$. In summary, we conclude that there is a $N_4>N_3$ so that \eqref{eq:4.13} holds for $n \geq N_4$. In particular, this shows that the image $\omega_n\left(\overline{\mathbb{D}}\right)$ will be bounded away from the point $\xi=1$ for such $n$.

      Now, we prove the statements (1) to (3). For $n \geq N_4$, by virtue of \eqref{eq:4.5} and \eqref{eq:4.6} we can re-write the equation $\Phi_n\left(\omega_n(z)\right)=z$ as follows: \[z^{2}[\eta_{\nu_n}(z)]^{k_n-2}=[\omega_n(z)]^{k_n}, \quad z \in \mathbb{D}.\] Moreover, Lemma 4.6 allows us to take the principal $(k_n-2)$-root of the last formula, and we obtain that \begin{equation}\label{eq:4.14}
             \eta_{\nu_n}(z)=\omega_n(z)[g_n(z)]^{2/(k_n-2)}, \quad z \in \overline{\mathbb{D}}.
            \end{equation} This shows that the function $\eta_{\nu_n}$ extends continuously to the circle $\mathbb{T}$. Also, by Lemmas 4.5 and 4.6 and \eqref{eq:4.14}, we have  \begin{equation}\nonumber
        \begin{split}
         \left|\omega_n(z) - \eta_{\nu_n}(z)\right|&=\left|\omega_n(z)\right|\left|1-[g_n(z)]^{2/(k_n -2)}\right|\\
                      & \leq \left|g_n(z)\right|^{2/k_n -2}\left|1- \exp\left(\frac{2 \arg g_n(z)}{k_n -2}i\right)\right|+1-\left|g_n(z)\right|^{2/(k_n -2)}\\ & \leq 2\left|\sin\left(2/(k_n -2)\right)\right|+1-C^{\, 2/(k_n -2)} \rightarrow 0
                             \end{split}
      \end{equation} uniformly for all $z \in \overline{\mathbb{D}}$ as $n \rightarrow \infty $. Combining this with \eqref{eq:4.12} and \eqref{eq:4.13}, we conclude that the ranges $\eta_{\nu_n}\left(\overline{\mathbb{D}}\right)$ are uniformly bounded away from the point $\xi=1$ for large $n$; more precisely, there exists $N_5>N_4$ such that \begin{equation}\label{eq:4.15} \left|1 - \eta_{\nu_n}(z)\right| \geq  \frac{1}{2}\text{d}[1,\eta_{\lambda_t}\left(\mathbb{T}\right)], \quad z \in \overline{\mathbb{D}}, \quad n \geq N_5.\end{equation} This implies that the expression \[p_n(z)=\frac{1}{2\pi}\frac{1-|\et{\nu_n}(z)|^2}{|1-\et{\nu_n}(z)|^2}\] defines a continuous function in the closed disc $\overline{\mathbb{D}}$ for $n \geq N_5$. Consequently, the absolute continuity of the measure $\nu_n$ and the continuity of its density $f_n$ for $n \geq N_5$ now follow at once from the continuity of $p_n$ on $\mathbb{T}$ and the Poisson type integral formula: \[p_n(z)=\frac{1}{2\pi}\int_{\mathbb{T}}\Re{\left[\frac{\xi + z}{\xi - z}\right]}\, d\nu_{n}\left(1/\xi\right), \quad z\in \mathbb{D}.\] Note that $f_n(\xi)=p_n(1/\xi)$ for $\xi \in \mathbb{T}$.

     Finally, the analyticity and the uniform convergence property for $f_n$ on the arc $I_{\varepsilon}$ are direct applications of Theorem 4.3. To finish the proof, we only need to show that the values of $f_n$ can be made uniformly small on the symmetric complement $\mathbb{T} \setminus I_{\varepsilon}$ as $n \rightarrow \infty$. Indeed, by \eqref{eq:4.13}, \eqref{eq:4.14} and \eqref{eq:4.15}, we have the estimate \[\limsup_{n\rightarrow \infty}\max_{\xi \in \mathbb{T} \setminus I_{\varepsilon}}\left|f_n(\xi)\right| \leq \,\frac{8d(\varepsilon)}{(\text{d}[1,\eta_{\lambda_t}\left(\mathbb{T}\right)])^2\pi}.\] By letting $\varepsilon \rightarrow 0$ (and hence $d(\varepsilon)\rightarrow 0$), the uniform convergence of $f_n$ now follows, and the proof is finished.     \end{proof}

   \section{convergence of free entropy and uniform universality}
      For each $\mu \in \mathcal{M}_{\mathbb{T}}$, its \emph{free entropy} $\Sigma(\mu)$ is defined through the logarithmic energy of $\mu$, that is, \[\Sigma(\mu)=\iint\limits_{\mathbb{T} \times \mathbb{T}}\log\left|\xi_1-\xi_2\right|\,d\mu(\xi_1)d\mu(\xi_2).\] This quantity was introduced and studied by Voiculescu in his theory of mutual free information \cite{DVV1999}. On the other hand, it was also proved in \cite{HPlarge} that the functional $\Sigma(\cdot)$ serves as the rate function in a large deviation theorem for the empirical eigenvalue distribution of random unitary matrices. The first application of our results concerns convergence of free entropy in limit theorems, whose proof follows easily from the dominated convergence theorem and the fact that the logarithmic integrand $\log\left|e^{i\theta_1}-e^{i\theta_2}\right|$ is integrable with respect to the product measure $d\theta_1 \times d\theta_2$ on the circle $\mathbb{T}$. \begin{prop}\label{prop:5.1}
      Suppose that $\{\mu_n\}_{n=1}^{\infty}$ is an infinitesimal sequence in $\mathcal{M}_{\mathbb{T}}$ and that the distributional approximation $\nu_n=(\mu_n)^{\boxtimes k_n} \Rightarrow \nu$ holds as $n\rightarrow \infty$. If the $\boxtimes$-infinitely divisible limit law $\nu$ has a positive density everywhere on $\mathbb{T}$, or, $\nu=\lambda_t$ for some $t>0$, then we have \[\lim_{n\rightarrow \infty}\Sigma(\nu_n)=\Sigma(\nu).\]
\end{prop}

A special case of the preceding result is the convergence of free entropy to that of the Haar measure $d\theta/2\pi$ on $\mathbb{T}$. It is well-known that the Haar measure has maximal free entropy $0$ among all probability laws on $\mathbb{T}$. Meanwhile, by Theorem 4.4, a discrete free convolution process $(\mu)^{\boxtimes n}$ of a non-degenerate law $\mu$ will flow to the Haar measure and the entropy is maximized over time. We conclude this paper with the following result which says that this uniform universality behavior also holds for continuous-time processes. Such a universality result was first proved for the free Brownian motion $U_t$ in Voiculescu's liberation process paper \cite{DVV1999} and later extended to free convolution semigroups in \cite{Zhong1}. Our approach here provides a conceptual proof of this phenomenon.
\begin{prop}\label{prop:5.2}
      Suppose that $\left(\nu_{t}\right)_{t \geq t_{0}}$ is a family of measures supported on $\mathbb{T}$ such that the initial distribution $\nu_{t_0}$ is non-degenerate, \[\nu_{t} \boxtimes \nu_{s}=\nu_{t+s},\quad s,t \geq t_{0},\] and $\nu_t \Rightarrow \nu_{t_0}$ as $t \rightarrow t_{0}^{+}$. Then, as $t \rightarrow \infty$, the flow $\nu_{t}$ converges to the Haar measure in law, in density, and in free entropy. In particular, every non-degenerate free unitary L\'{e}vy process flows to the Haar measure in this way as the time parameter tends to infinity.\end{prop}
      \begin{proof}Note first that the measure $\nu_{t_0}$ is not a point mass if and only if $\left|\text{m}(\nu_{t_0})\right|<1$. If the first moment $\text{m}(\nu_{t_0})=0$, then the definition of freeness implies $\nu_{2t_0}=\nu_{t_0} \boxtimes \nu_{t_0}=d\theta/2\pi$. Moreover, for any $t>t_0$, one has $\nu_t=\nu_{t-2t_0} \boxtimes d\theta/2\pi=d\theta/2\pi$ from the property of vanishing mixed free cumulants, and hence the flow $\nu_t$ terminates at the Haar measure when $t=2t_0$.

In the case of nonzero first moment, let us assume that $c>\left|\text{m}(\nu_{t_0})\right|>1/c$ for some constant $c \in (0,1)$, and let $\{t_n\}_{n=1}^{\infty}$ be an arbitrary sequence of real numbers such that $t_0 \leq t_n \rightarrow \infty$ as $n\rightarrow \infty$. In this case, we introduce \[\mu_{n}=\nu_{t_n/k_n} \quad \text{and} \quad k_n=[t_n/t_0],\quad n \geq 1,\] where $[x]$ denotes the integer part of $x\in \mathbb{R}$. Since $t_n/k_n \rightarrow t_0^{+}$ as $n \rightarrow \infty$, we conclude that $c>\left|\text{m}(\mu_n)\right|>1/c$ for sufficiently large $n$. This implies \[\lim_{n \rightarrow \infty}\left[\text{m}(\mu_n)\right]^{k_n-2}=0,\] whence the desire results follow from Theorem 4.4 and the fact that $\nu_{t_n}=(\mu_n)^{\boxtimes k_n}$.
       \end{proof}

 \subsection*{Acknowledgement}The third named author would like to thank Professors Hari Bercovici and Vittorino Pata for thoughtful conversations. The first named author was partially supported by NSF grant DMS-1160849, and the second named author was supported by an NSERC Canada Discovery Grant.

\bibliography{clipboard}

\bibliographystyle{amsplain}

\end{document}